\documentclass{amsart}
   
\newcommand{\comments}[1]{}
\usepackage{currfile}
\usepackage{amssymb}

\usepackage{xcolor}
\newcommand{\Cal}{\mathcal}

\theoremstyle{plain}
\newtheorem{theorem}{Theorem}[section]
\newtheorem{corollary}{Corollary}[section]

\newtheorem{lemma}{Lemma}[section]
\newtheorem{proposition}{Proposition}[section]

\theoremstyle{definition}

\newtheorem{example}{Example}[section]
 
\theoremstyle{remark}
\newtheorem{question}{Question}[section]
\newtheorem{remark}{Remark}[section]

\numberwithin{equation}{section}

\begin{document}

\title[Mixed character sums]{Mixed character sums modulo prime powers}
\author{Todd Cochrane}
\address{ Department of Mathematics\\
          Kansas State University\\
          Manhattan, KS 66506}
\email{cochrane@math.ksu.edu}

\author{Andrew Granville}
\address{Department de Math\'ematiques et Statistique, Universit\'e de Montr\'eal\\
CP 6128  Centre-Ville\\ Montreal, QC H3C 3J7, Canada}
\email{andrew@dms.umontreal.ca}

\thanks{AG is partially supported by a grant from the the Natural Sciences and Engineering Research Council of Canada.}
   
\dedicatory{Dedicated to Roger Heath-Brown on the occasion of his 75th birthday}

\date{\today}

\begin{abstract}
We obtain explicit estimates for the mixed character sum $$S(\chi,g,f,p^m) = \sum_{x=1}^{p^m}
 \chi (g(x)) e_{p^m}(f(x)),
 $$
where $p^m$ is a prime power, $\chi$ is a
multiplicative character mod $p^m$ and $f,g$ are rational functions over $\mathbb Q$. Let $f=f_+/f_-$, $g=g_+/g_-$ in reduced form, and set $D=\deg(f)+Z-1$ where $Z$ is the number of distinct complex zeros of $f_-g_+g_-$, and $\Delta= \deg(f)+\deg(g)$ for polynomial $f,g$, $\Delta=2(\deg(f)+\deg(g))$ otherwise. We show for example that for odd $p$,  any non-degenerate sum has
\begin{equation} 
    |S(\chi,g,f,p^m)| \le \begin{cases}
3^{4/3}\, p^{m(1-\frac 1D)}, & \text{if $\deg_p(f) \ge 1$;}\\    
     3^{4/3}\, p^{m(1-\frac 1\Delta)}, & \text{if $\deg_p(g) \ge 1$.}
     \end{cases}
     \end{equation}
Analogous bounds are given for degenerate sums as well.

\end{abstract}
\maketitle

\begin{small}

\section{Introduction} \label{S:1}

Uniform bounds on exponential sums of the sort 
$$
\Big|\sum_{x=1}^{p^m} e_{p^m}(f(x))\Big| \ll_{\varepsilon,d} p^{m\left(1-\frac 1{d}+\varepsilon\right)},
$$
 with $f$ a polynomial of degree $d$, $p^m$ a prime power and $e_{p^m}(\cdot)$ the additive character mod $p^m$, 
$
e_{p^m}(x)=e^{\frac {2\pi ix}{p^m}},
$
date back to the work of Mordell \cite{mordell} and Hua \cite{hua1937,hua1938,hua1940}. Numerous extensions and refinements have been given since then.  
In this paper we obtain improved estimates  of this type for a general  mixed character sum, 
\begin{equation} \label{sum}
S(\chi,g,f,p^m) = \sum_{x=1}^{p^m}
 \chi (g(x)) e_{p^m}(f(x)),
\end{equation}
where $\chi$ is a
multiplicative character
mod ${p^m}$,
and $f$ and $g$ are rational functions over $\mathbb Q$,
\begin{equation} \label{fg}
f(x)=\frac {f_+(x)}{f_-(x)}, \quad \quad g(x)=\frac {g_+(x)}{g_-(x)},
\end{equation}
for some relatively prime polynomials $f_+, f_-$, and $g_+,g_-$ over $\mathbb Z$. The sum in \eqref{sum} is over values of $x$ for which 
$ 
p \nmid f_-(x)g_+(x)g_-(x).
$ 
For such $x$, $f(x):\equiv f_+(x)\overline{f_-(x)}$ mod $p^m$, 
$g(x):\equiv g_+(x)\overline{g_-(x)}$ mod $p^m$
 with the overline denoting multiplicative inverse mod $p^m$.

Let $\deg(f):=\max\{\deg(f_+),\ \deg(f_-)\}$, let   $\deg_p(f)$ be the degree of $f$ reduced mod $p$, and set
\begin{equation} \label{Ddef}
D:=  \deg(f) + \mathcal Z(f_-g_+g_-),
\end{equation}
where $\mathcal Z(h)$ denote the number of distinct complex zeros of any polynomial $h\in \mathbb Q[x]$.

For prime moduli (that is, when $m=1$)  ``best-possible''  results are known, stemming from the famous work of Weil \cite{weil}.
First note that there are  {\it degenerate} cases when  $\deg_p(f)=0$ and $g(x) \equiv bh(x)^r$ mod $p$ for
  some  $b\in \mathbb Q$ and  $h(x) \in \mathbb Q(x)$, where $r$ is the order of $\chi$, in which case  $\chi(g(x))e_{p}(f(x))$ is  constant when it is non-zero so there is no cancelation.  For a non-degenerate mod $p$ sum we do have significant cancelation:
\begin{equation} \label{weilub}
|S(\chi,g,f,p)| \le (D-1) \sqrt{p};
\end{equation}
and $D$ can be replaced by $D_p:=\deg_p(f)+\mathcal Z_p(f_-g_+g_-) \le D$, where $\mathcal Z_p$ denotes the number of distinct zeros over $\overline{\mathbb F}_p$.
 For the case of polynomials we refer the reader to the works of Schmidt
\cite{schmidt} and Stepanov \cite{stepanov} for elementary proofs of \eqref{weilub}, and
for rational functions, 
Bombieri \cite{bombieri}, Perelmuter \cite{perelmuter} and Cochrane and Pinner \cite{cp}. 

Our contribution in this paper comes in the cases where $m\geq 2$. Here  the sum $S(\chi,g,f,p^m)$ is  {\it degenerate} when
$\deg_p(f)=0$, and either $\deg_p(g)=0$ or $\chi$ is imprimitive, in which case
 $S(\chi,g,f,p^m)$ reduces to a multiple of a complete sum mod $p^{m-\ell}$ for some $\ell \ge 1$. 
Our main theorem is the following.

\begin{theorem}\label{maintheorem0} Suppose that $f(x),g(x)$ are rational functions over $\mathbb Q$, $p$ is a prime, $m$ a positive integer, and
$\chi$ is a multiplicative character mod ${p^m}$ such that $S(\chi,g,f,p^m)$ is non-degenerate.

 i) If $\deg_p(f) >0$, then for odd $p$,
\begin{equation} \label{uniformf}
|S(\chi,g,f,p^m)|\le 3\, \deg_p(f)^{\frac 1D}\, p^{m(1-\frac 1{D})}, 
\end{equation}
and the same for $p=2$, with $3$ replaced by $2^{\frac 53}$.

ii)
If $\chi$ is primitive and $\deg_p(g)>0$,  then
\begin{equation} \label{uniformg}
|S(\chi,g,f,p^m)|\le \max\{3\, \deg_p(g)^{\frac 1D},\, \deg_p(g)^{\frac 2D}\}\, p^{m(1-\frac 1D)},
\end{equation}
for odd $p$, and for $p=2$, $|S(\chi,g,f,2^m)| \le 2^{\frac 53} \deg_2(g)^{\frac 1D} 2^{m(1-\frac 1D)}$.
\end{theorem}

 This is the first time we are aware of where the  upper bound has been stated in terms of the parameter $D$ instead of 
\begin{align} 
\Delta:&=\begin{cases} \deg(f)+\deg(g), & \text{for polynomial $f,g$;}\\
  2\deg(f)+2\deg(g),& \text{if $f$ or $g$ is non-polynomial,}
  \label{Deltadef}
  \end{cases}
 \end{align}
which potentially gives a saving of a power of $p$ as $D \le \Delta$, and is smaller in many cases.

     The exponent $m(1-\frac 1D)$ in \eqref{uniformf} and \eqref{uniformg} is best possible, which is well known for sums such as $\sum_{x=1}^{p^m}e_{p^m}(ax^d)$, where $D=d$ (see example \ref{Heilbronnex}, and for further extremal examples see  \cite{cz2}).

For a degenerate  sum that reduces to a multiple of a   complete sum of the type \eqref{sum} mod $p^{m-\ell}$, the estimate in Theorem \ref{maintheorem0}  must be increased by a factor of $p^{\frac {\ell}{D}}$, as explained in
 Theorem \ref{degentheorem} and Proposition \ref{degenprop}.

The following slightly cleaner bounds might be more useful in applications: 

\begin{corollary} \label{maincor} For any non-degenerate sum $S(\chi,g,f,p^m)$ we have for odd $p$,
\begin{equation} \label{uuniform} 
    |S(\chi,g,f,p^m)| \le \begin{cases}
3^{4/3}\, p^{m(1-\frac 1D)}, & \text{if $\deg_p(f) \ge 1$;}\\    
     3^{4/3}\, p^{m(1-\frac 1\Delta)}, & \text{if $\deg_p(g) \ge 1$;}
     \end{cases}
     \end{equation}
and the same for $p=2$ with $3^{4/3}$ replaced by $2^{5/3}3^{1/3}$.
\end{corollary}

This paper refines and extends the work of the first author \cite{cochrane2002} where similar but weaker bounds were established for the case of polynomials,  and corrects a couple errors that were made in that work; see Remarks \ref{errorremark} and \ref{errorremark2}.  The case of degenerate sums was not addressed in \cite{cochrane2002}.  
The motivation for making these improvements was the application of these bounds to incomplete mixed character sums in the authors' recent work \cite{cgz}. The work here builds upon earlier work by many authors; see \cite{czpuremix, cz2, czsurvey, cochrane2002, cochrane2024} for more background and references.

\section{notation and background}
  We follow the convention of using upper case $X$ if $f(X)$ is to be viewed as a formal object in an indeterminate $X$, and lower case $x$ if $f(x)$ is to be  viewed as a function.
  
For any rational function $h(X)$ over $\mathbb Q$, $h_+(X)$ and $h_-(X)$ are defined to be relatively prime integer polynomials such that $h(X)=h_+(X)/h_-(X)$.

 We say that $\chi(g(x))e_{p^m}(f(x))$ is {\it constant on its domain} if it is constant on the set of integers $x$ for which $
p \nmid f_-(x)g_+(x)g_-(x)
$.

Let $\text{ord}_p(x)$ denote the $p$-adic valuation on $\mathbb Z$, $\text{ord}_p(x)=k$  if $p^k\|x$; $\text{ord}_p(0):=\infty$.  The valuation extends to the ring of $p$-adic integers $\mathbb Z_p$ and $p$-adic field $\mathbb Q_p$, and to the field of rational functions over $\mathbb Q$. Thus for a polynomial $f$ over $\mathbb Z$, 
$\text{ord}_p(f)$ is the largest power of $p$ dividing all of the coefficients
of $f$, and for a rational function $f=f_+/f_-$ ,  $\text{ord}_p(f) =
\text{ord}_p(f_+)- \text{ord}_p(f_-)$. We write $p^k|f(X)$ if $\text{ord}_p(f) \ge k$ and  $p^k\|f(X)$ if $\text{ord}_p(f)=k$.

The sum in \eqref{sum} is empty if $p|f_-(X)$ or $p|g_+(X)g_-(X)$, and so we always assume that $p \nmid f_-(X)g_+(X)g_-(X)$, that is $\text{ord}_p(f) \ge 0$ and $\text{ord}_p(g) = 0$.

 Assume now that $p$ is odd; see Section \ref{p2section} for $p=2$.
Let
$\omega$ be a primitive root mod ${p^2}$ and $r$ the integer defined by
$\omega^{p-1}=1+rp$.
In particular, $p \nmid r$ and
 $\omega$ is a primitive root mod ${p^m}$ for any exponent
$m$.
For any
multiplicative character $\chi$ mod ${p^m}$ there is a unique integer $c$
 with $1 \le c \le p^{m-1}(p-1)$ such that for any $k \in \mathbb Z$,
\begin{equation}  \label{c}
\chi(\omega^k)=e^{\frac {2\pi i ck}{p^{m-1}(p-1)}}.
\end{equation}
Define
\begin{equation}\label{tchidef}
t_\chi:=\text{ord}_p(c),
\end{equation}
so that $p^{m-t_\chi}$ is the conductor of $\chi$, and $\chi$ is primitive if and only if $t_\chi=0$. For the principal character $\chi_0$ we have  $c=p^{m-1}(p-1)$.
If $\omega$ is replaced by  another primitive root, say $\omega^l$ with $(l,p(p-1))=1$, then $c$ is replaced by $c'\equiv cl$ mod $p^{m-1}(p-1)$, and $t_\chi$ remains the same.

In this paragraph we focus on the set of integers congruent to 1 mod $p$. The $p$-adic logarithm  is defined on this set  by
$$  
\log(1+py) =  \sum_{n=1}^\infty (-1)^{n-1}\frac {(py)^n}n.
$$
The series converges for any $y \in  \mathbb Z_p$ since $n-\text{ord}_p(n) \to \infty$ as $n \to \infty$. Also, if $y \equiv y'$ mod $p^{m-1}$ then $\log(1+py) \equiv \log(1+py')$ mod $p^m$.

 Let $R$
denote the $p$-adic
integer
$R:= p^{-1}\log(1+rp)=p^{-1} \log(\omega^{p-1})$,
and $\overline R$ its multiplicative inverse in $\mathbb Z_p$ so that $R\overline R \equiv 1$ mod $p^m$ for any positive integer $m$. 
If $\omega$ is replaced by $\omega^l$ then  $R$ is replaced by $R'=Rl = p^{-1}\log (\omega^{l(p-1)})$, and we see that $\overline Rc\equiv \overline R'c'$ mod $p^{m-1}(p-1)$. Thus  $\overline Rc$ is an invariant associated with $\chi$ and we define
\begin{equation} \label{cchidef}
c_\chi:\equiv \overline Rc \pmod {p^{m-1}}, \quad \quad 1 \le c_\chi \le p^{m-1}.
\end{equation}
 For any integers $j,y$ with $1+py \equiv \omega^{j(p-1)}$ mod $p^m$  we have  $\log(1+py) \equiv j\log(\omega^{p-1})\equiv pRj$ mod $p^m$ and 
\begin{equation} \label{plog}
\chi(1+py)= e_{p^{m-1}}(cj)=e_{p^m}(c_\chi\log(1+py)).
\end{equation}
Note, the value of the exponential just depends on the residue class of $c_\chi$ mod $p^{m-1}$.

For any rational functions $f,g$ over $\mathbb Q$, not both constant polynomials, $g \neq 0$,  and multiplicative character $\chi$ mod $p^m$,  we define
\begin{align} 
t&=t(\chi,g,f,p^m):= \text{ord}_p\Big(f'(X) +c_\chi \frac {g'(X)}{g(X)}\Big)\label{t},
\end{align}
and the {\it critical point function}
\begin{align}
\mathcal C(X)&:= p^{-t}\Big(f'(X) +c_\chi\frac {g'(X)}{g(X)}\Big). \label{cpfunction}
 \end{align}
By Lemma \ref{Fieldlemma}, $f'+c_\chi g'/g$ is not identically zero, and so $t$ is well defined.

\section{Degenerate sums}
Suppose that $m \ge 2$ and that the sum $S(\chi,g,f,p^m)$ is degenerate, that is $\deg_p(f)=0$ and either $\deg_p(g)=0$ or $\chi$ is imprimitive.   
Write
\begin{equation} \label{FGintro}
f(X)=f(0)+p^{\ell_f}F(X), \quad \quad g(X)=g(0)(1+p^{\ell_g}G(X)),
\end{equation}
for some non-negative integers $\ell_f$, $\ell_g$ and rational functions $F,G$ with $p \nmid FG$. We may assume that $p \nmid g(0)$,  $\deg_p(F) \ge 1$ and $\deg_p(G) \ge 1$ (see Section \ref{degensection}).   
For a degenerate sum we have both $\ell_f >0$ and $\ell_g+t_\chi>0$. 
If $\ell_f$ and $\ell_g$ are both positive, we define $H(X)$ to be a rational function with
\begin{equation} \label{Hdef}
H(X):\equiv p^{\ell_f}F(X) + c_\chi \log(1+p^{\ell_g}G(X)) \pmod {p^m},
\end{equation} 
where $\log$ denotes the $p$-adic logarithm.

 Set $\ell=m$ if $\ell_g=0$,  $\ell_f\ge m$, $t_\chi=m-1$ and $g(X)\equiv bh(X)^r$ mod $p$ for some rational function $h(X)$, where $r$ is the order of $\chi$, as well as when $\ell_g>0$ and $\text{ord}_p(H) \ge m$. 
 In these cases,   $\chi(g(x))e_{p^m}(f(x))$ is constant on its domain (see \eqref{chiH}), that is, $S(\chi,g,f,p^m)$ degenerates to a mod 1 sum and no nontrivial estimate is available.   Otherwise, set
\begin{equation} \label{eldef}
\ell=\ell(\chi,g,f,p^m) := \begin{cases} \min\{\ell_f, t_\chi\}, & \text{if $\ell_f>0$, $\ell_g=0$};\\
\text{ord}_p(H(X)), & \text{if $\ell_f>0$, $\ell_g>0$.}
\end{cases}
\end{equation}
In these cases, $1 \le \ell \le m-1$ (for degenerate sums).
Note also that if $\ell_f \neq \ell_g+t_\chi$, then $\ell = \min\{\ell_f, \ell_g+t_\chi\}$ (in both cases).

\begin{theorem}\label{degentheorem} 

Let $f,g$ be rational functions over $\mathbb Q$, not both constants, $p^m$ a prime power, $\chi$ a  multiplicative character  mod $p^m$,  $\ell, t$  as defined  in \eqref{eldef}, \eqref{t} and $\Delta$ as defined in \eqref{Deltadef}. 

i) If $\ell_g=0$ then $S(\chi,g,f,p^m)$ degenerates to a  mixed exponential sum mod $p^{m-\ell}$ and 
\begin{equation} 
|S(\chi,g,f,p^m)|  \le 3^{\frac 43} p^{\frac {\ell}\Delta} p^{m(1-\frac 1{\Delta})},
\end{equation}  
with $3^{\frac 43}$ replaced by $2^{\frac 53}3^{\frac 13}$ for $p=2$.

ii) If $\ell_g>0$ then  $S(\chi,g,f,p^m)$ degenerates to a  pure exponential sum mod $p^{m-\ell}$, and with $d_p:=\deg_p(p^{-\ell} H)$,
\begin{equation} \label{degenublg1}
|S(\chi,g,f,p^m)|  \le \max\{3 \, d_p^{\frac 1\Delta}, d_p^{\frac 2{\Delta}}\} p^{\frac {\ell}\Delta} p^{m(1-\frac 1{\Delta})}.
\end{equation} 
   \end{theorem}
\noindent We show, equation \eqref{dpHub}, that $d_p \le \deg(F)+ \tfrac {\ell-t_\chi}{\ell_g} \tfrac p{p-1} \deg(G)$, but this can likely be improved.  
  
For degenerate sums, the value of $\chi(g(x))e_{p^m}(f(x))$ depends only on the residue class of $x$ mod $p^{m-\ell}$. 
This can occur as well for certain non-degenerate sums as the following example illustrates.

\begin{example} \label{Heilbronnex} Consider $\sum_{x=1}^{p^m} e_{p^m}(x^{p^t})$ with $1 \le t <m$.  The value of $e_{p^m}(x^{p^t})$ depends only on the residue class of $x$ mod $p^{m-t}$, and so $\sum_{x=1}^{p^m} e_{p^m}(x^{p^t})=p^t\sum_{x=1}^{p^{m-t}} e_{p^m}(x^{p^t})$, but the latter sum is not a complete  exponential sum.  This sum is  non-degenerate (by our definition) but the estimate in Theorem \ref{maintheorem0} is only nontrivial for $m$ very large relative to $t$, specifically $p^m>p^t 3^{p^t}$. By repeated application of  Proposition \ref{propconvert} one can  obtain a more precise estimate. For $m \not\equiv 1$ mod $d$, one can show that 
$
\Big|\sum_{x=1}^{p^m} e_{p^m}(x^{p^t})\Big|= p^{m-\lceil m/d\rceil};
$
see \cite[Example 9.1]{czpuremix}. On the other hand, 
if $m \equiv 1$ mod $d$, on the final application one is left with a  generalized Heilbronn sum of the sort $\sum_{x=1}^{p^m}e_{p^m}(x^{p^{m-1}})$. Nontrivial bounds are known for such sums, but
this paper has nothing to say about them. 
\end{example}

 The following are examples of the two cases where $\ell=m$. 
   
 \begin{example}
i) Let $\chi$ be a mod $p^m$ character of conductor $p$ and order $r$. Then for any $f,g$ of the sort $g(X)=G(X)^r$, $f(X)=p^mF(X)$, with $G,F$ arbitrary rational functions, we have $\ell=m$ and  $\chi(g(x))e_{p^m}(f(x))=1$  on its domain. 
 
ii) Let $g(X)=1+p\, G(X)$, with $G(X)$ any rational function, $\chi$ any character mod $p^m$  and $f(X)$ any rational function with $f(X) \equiv -c_\chi \log(1+p\, G(X))$ mod $p^m$. Then $\ell=m$ and by \eqref{plog},  $\chi(g(x))e_{p^m}(f(x))= 1$ on its domain.
\end{example}

For certain Laurent polynomials we can refine the bound in Theorem \ref{degentheorem} $(ii)$. 
  
\begin{theorem} \label{Laurenttheorem} Let $f,g$ be rational functions as given in  \eqref{FGintro} such that $F$ and $G$ are Laurent polynomials, say $G(X)=\sum_{j=d_1}^{d_2}a_jX^j$. Let $p^m$ be an  odd prime power.  If $|d_2|\ge |d_1|$ assume $p \nmid a_{d_2}$. If $|d_2|<|d_1|$ assume $p \nmid a_{d_1}$. The coefficients of $F$ are arbitrary.
 Then for any  character $\chi$ mod $p^m$ we have
\begin{equation} \label{degenubLaur}
|S(\chi,g,f,p^m)| \le 6\ p^{\frac {\ell}\Delta} p^{m(1-\frac 1\Delta)},
\end{equation}
with $6$ improved to $3^{4/3}$ for non-polynomial $f,g$. 
If $p=2$ the same holds with $6$ replaced by  $2^{\frac 83}$, and improvement to $2^{\frac 53}3^{\frac 13}$ for non-polynomial $f$ or $g$.
\end{theorem} 

\begin{remark} The estimate in \eqref{degenubLaur} holds as well in the following three degenerate cases: If $\ell_g=0$; if $\ell_f \neq \ell_g+t_\chi$; if $p>\deg_p(p^{-\ell}H)$.
\end{remark} 

\begin{question} Does a bound of the sort \eqref{degenubLaur} (with some absolute constant) hold for any degenerate sum with $\ell_g>0$?
\end{question}

\section{Background on Prime moduli}

We need the following consequence of the Weil bound \eqref{weilub}:
\begin{equation} \label{weilub2}
|S(\chi,g,f,p)| \le 1.75\, p^{1-\frac 1D},
\end{equation}
valid for any non-degenerate sum.
Indeed, if $p^{\frac 1D}\le 1.75$, then trivially, $|S(\chi,g,f,p)|\le p< 1.75\, p^{1-\frac 1D}$. 
If $p^{\frac 1D}>1.75$, then by the Weil bound,
$$
|S(\chi,g,f,p)| \le (D-1)\ \sqrt{p}< 1.75\  p^{1-\frac 1D},
$$
since $(D-1)^{2/D} \le 4^{\frac 25}= 1.741...$ for $D \ge 1$. The bound in \eqref{weilub2} establishes Theorem \ref{maintheorem0} (with a stronger constant) for the case of prime moduli.

\section{Main results on Prime power moduli}
We restrict our attention to odd $p$ in this section and take up the
case $p=2$ in Section \ref{p2section}. 
Let $f,g$ be as given in \eqref{fg}, $\chi$ a multiplicative character mod $p^m$, and $t=t(\chi,g,f,p^m)$ and $\mathcal C(X)$ be as given in \eqref{t}, \eqref{cpfunction}.  
  We define the set of {\it critical points} $\Cal A \subseteq \mathbb F_p$
associated with the sum $S(\chi,g,f,p^m)$ to be the set of solutions to the {\it
critical point congruence},
\begin{equation} \label{cpc}
 \Cal C(x)= p^{-t}\Big(f'(x) +c_\chi\, \frac {g'(x)}{g(x)}\Big)
\equiv 0 \pmod p.
\end{equation}
 Thus,
\begin{equation} \label{Adef}
\Cal A:= \{\alpha \in \mathbb F_p: \Cal C(\alpha ) \equiv 0 \pmod p  \}.
\end{equation}
 An integer is called a critical point if its residue class mod $p$ is one. A critical point is said to have multiplicity $\nu$ if it is a zero of \eqref{cpc} of multiplicity $\nu$.  Write $\mathcal C(X) = \frac {\mathcal C_+(X)}{\mathcal C_-(X)}$ in reduced form with $\mathcal C_+, \mathcal C_-$ integer polynomials.

Write $S(\chi,g,f,p^m)= \sum_{\alpha =1}^{p} S_\alpha$ with
\begin{equation} \label{salph}
S_\alpha = S_\alpha(\chi,g,f,p^m):= \sum_{\substack{x=1\\x\equiv \alpha \pmod
p}}^{p^m} \chi (g(x)) e_{p^m}(f(x)).
\end{equation}
Our main results for prime power moduli are the following.

\begin{theorem} \label{maintheorem} Let $f,g$ be rational functions over $\mathbb
Z$, not both constants, $p$ be an odd prime, and $\chi$ a multiplicative
character mod ${p^m}$ with
 $m \ge t +2$. 

\noindent
i) If $\alpha \notin \Cal A$,  then $S_\alpha(\chi,g,f,p^m)=0$.

\noindent
ii) If $\alpha$ is a critical point of multiplicity one, then $|S_\alpha(\chi,g,f,p^m)|=p^{\frac {m+t}2}$.

\noindent
iii) If $\alpha$ is a critical point of multiplicity $\nu \ge 1$, then
\begin{equation} \label{mainlocalub}
|S_\alpha(\chi,g,f,p^m)| \le 1.75\, p^{\frac t{\nu+1}}p^{m(1-\frac 1{\nu+1})}.
\end{equation}
\end{theorem}
\noindent

\begin{corollary} \label{maincor1} Under the hypotheses of Theorem \ref{maintheorem}, we have
\begin{equation} \label{maincorub1}
|S(\chi,g,f,p^m)| \le 3\, p^{\frac t{\deg_p(\mathcal C_+)+1}} p^{m\big(1-\frac 1{\deg_p(\mathcal C_+) +1}\big)}.
\end{equation}
\end{corollary}

\noindent In particular, since $\deg_p(\mathcal C_+) \le D-1$, as we show in Lemma \ref{CDub}, we readily deduce the following.

\begin{corollary} \label{m2} Let $f,g$ be rational functions over $\mathbb
Z$, not both constants, $p$ be an odd prime, and $\chi$ a multiplicative
character mod ${p^m}$. Then for $m \ge 1$, 
$$
|S(\chi,g,f,p^m)| \le \begin{cases} p^{\frac {t+1}D} p^{m(1-\frac 1D)}, & \text{if $m \le t+1$;}\\ 3p^{\frac tD} p^{m(1-\frac 1D)}, & \text{if $m \ge t+2$.}\end{cases}
$$
\end{corollary}

\begin{remark} We note that the results of this section, Theorem \ref{maintheorem}, Corollary \ref{maincor1} and Corollary \ref{m2} hold regardless of whether $S(\chi,g,f,p^m)$ is degenerate or not.
\end{remark}

\section{Upper bound for $t(\chi,g,f,p^m)$}
For any prime $p$, positive integer $k$, put $N_k:=k+1+\lfloor \log_p k\rfloor$, and observe that
$$
\log(1+pX)=\sum_{n=1}^\infty  (-1)^{n+1} \frac {(pX)^n}{n}\equiv \sum_{n=1}^{N_k} (-1)^{n+1} \frac {p^n}n X^n 
 \pmod {p^k},
 $$
  so that $\log(1+pX)$ can be viewed as a polynomial over $\mathbb Z$, mod $p^k$. 
 For
 any rational function $u(X)$ over $\mathbb Q$ with $\text{ord}_p(u(X)) \ge 0$, the function
$\log(1+pu(X))$ 
can in the same manner be viewed as a rational function over $\mathbb Z$, mod $p^k$.

\begin{proposition} \label{tprop} Let $f,g$ be rational functions over $\mathbb Q$, $p$ a prime with  $\text{ord}_p(f) \ge 0$ and $\text{ord}_p(g) \ge 0$, and $k$ a positive integer such that $p^k|(f'(X)+g'(X)/g(X))$. Then there exist  rational functions $u(X)$, $g_0(X)$ and  $f_i(X)$, $0 \le i \le k$  such that $\text{ord}_p(u(X)) \ge 0$ and
\begin{align}
f(X) & \equiv  \sum_{i=0}^{k-1} p^if_i(X^{p^{k-i}}) -\log(1+pu(X)) \pmod {p^k};\label{fsp}\\
g(X) &\equiv g_0(X^{p^k})(1+pu(X))\pmod {p^k}.\label{gsp}
\end{align}
\end{proposition}

\noindent The nonzero coefficients of the $f_i$, as well as $g_0$ may be taken to be nonzero integers mod $p$.  The converse is immediate. Indeed if $f,g$ are as in \eqref{fsp} and \eqref{gsp}, then
\begin{equation}
f'(X)+\frac {g'(X)}{g(X)} \equiv \frac {-pu'(X)}{1+pu(X)} +\frac {pu'(X)}{1+pu(X)} \equiv 0 \pmod {p^k}.
\end{equation}
Taking $g(X)=1$ gives the following corollary.

\begin{corollary}\label{fspcorollary} Let $f$ be a rational function over $\mathbb Q$, $p$ a prime with $\text{ord}_p(f) \ge 0$, and $k$ a positive integer such that  $p^k|f'(X)$. Then there exist  rational functions  $f_i(X)$, $0 \le i \le k-1$ such that
\begin{align}
f(X) & \equiv  \sum_{i=0}^{k-1} p^if_i(X^{p^{k-i}}) \pmod {p^k}. \label{fspcor}
\end{align}
In particular,  $p^k|\deg_p(f)$.
\end{corollary}

We readily deduce from the proposition the following.

\begin{corollary} \label{tboundcor} Let $f,g$ be rational functions over $\mathbb Q$, not both constants, $p$ a prime with $\text{ord}_p(f) \ge 0$ and $\text{ord}_p(g) \ge 0$, $\chi$ be a multiplicative character mod $p^m$ with $p^{t_\chi}\|c_\chi$ and  
$t=t(\chi,g,f,p^m)$ be as defined in \eqref{t}.  Then 
\begin{align*}
f(X) &\equiv f_0(X)^{p^t} \pmod p,
\end{align*}
for some rational function $f_0$ over $\mathbb Q$, and if $t \ge t_\chi$, then
 \begin{align*}
g(X) & \equiv g_0(X)^{p^{t-t_\chi}} \pmod p, 
\end{align*}
for some rational function $g_0$ over $\mathbb Q$. In particular,  
\begin{equation} \label{tdiv}
p^t|(\deg_p(f), \, c_\chi \deg_p(g)).
\end{equation}
\end{corollary}
\noindent
  
\begin{proof}  Put $G(X):=g(X)^{c_\chi}$, so that
$f'+c_\chi g'/g= f'+G'/G$. Thus $p^t\| (f'+G'/G)$ and we obtain from the proposition that
$f(X) \equiv f_0(X^{p^t})\equiv f_0(X)^{p^t}$ mod $p$ and $G(X) \equiv G_0(X^{p^t})\equiv G_0(X)^{p^t}$ mod $p$ for some rational functions $f_0,G_0$.
Thus $p^t|\deg_p(f)$. 

Suppose now that $q(X)$ is an irreducible factor of $g(X)$ over $\mathbb F_p$ with $q(X)^e\|g(X)$, so that $q(X)^{c_\chi e} \|G(X)$ that is $q(X)^{c_\chi e}\|G_0(X)^{p^t}$ over $\mathbb F_p$. Then it follows that $p^t|c_\chi e$ and so 
either $p^t|c_\chi$, or $t > t_\chi$ and $p^{t-t_\chi}|e$. Since this holds for every irreducible factor of $g(X)$ it follows that either $p^t|c_\chi$, or $g(X)=g_0(X)^{p^{t-t_\chi}}$ over $\mathbb F_p$ for some $g_0(X)$, and $p^{t-t_\chi}|\deg_p(g)$. In both cases we have $p^t|c_\chi \deg_p(g)$.
\end{proof}

\begin{remark} \label{errorremark} In \cite{cochrane2002}, Lemma 2.2 asserted incorrectly that $t=\min\{\text{ord}_p(f'),\, \text{ord}_p(cg')\}$, from which we deduced the result in Corollary \ref{tboundcor} above.  The proof given for Lemma 2.2 implicitly  assumed that the zeros of the distinct irreducible factors of $f$ and $g$ were all distinct mod $p$.  The next example provides $f$ and $g$ where $t$ can be arbitrarily larger than $\text{ord}_p(f')$ and $\text{ord}_p(cg')$.
\end{remark}

\begin{example} For any positive integers $T \le t$, let $f(X)=X^{p^t}+pX^{p^{T-1}}$, and $g(X)$ be a polynomial with $g(X)\equiv X^{p^t}e^{-pX^{p^{T-1}}}$ mod $p^{t+1}$. Then 
$$ 
f'+g'/g \equiv p^t X^{p^t-1} + p^{T}X^{p^{T-1}-1} + \frac {p^t}{X} -p^T X^{p^{T-1}-1} \equiv p^t\Big( X^{p^t-1} +  \frac {1}{X}\Big) \pmod {p^{t+1}}.
$$
Thus we have $p^T\|f'(X)$, $p^T\|g'(X)$ and $p^t\| (f'+g'/g)$.
\end{example}

\section{Proof of Proposition \ref{tprop}}
\begin{lemma}\label{Fieldlemma}  Let $f,g$ be rational functions over a field $F$ with $g \neq 0$ and $c \in F$. Then $f'+cg'/g=0$ if and only if $f'=0$ and $cg'=0$.
\end{lemma} 

\begin{proof} Suppose that $f'+cg'/g=0$ and consider the equation
$
f' = -cg'/g
$
together with the factorizations of both sides
over the algebraic closure of $F$. If $cg'/g$ is nonzero, then it has a simple pole at some zero or pole of $g$. On the other hand, $f'$ has no simple pole. Thus $cg'/g$ must be identically zero, implying that $cg'=f'=0$. The converse is trivial.
\end{proof}

\begin{lemma} \label{fprimefplemma} If $f$ is a rational function over $\mathbb F_p$ such that $f'=0$ identically, then $f(X)=F(X^p)$ for some rational function $F$ over $\mathbb F_p$.
\end{lemma}

\begin{proof} Let $f$ have a formal Laurent expansion $f(X)=\sum_{n=n_0}^\infty a_nX^n$ for some integer $n_0$ and $a_n \in \mathbb F_p$.  Since $f$ is rational, the coefficients $\{a_n\}_{n=L}^\infty$ satisfy a linear recurrence. Then $f'=0$ implies that $a_nn=0$ for $n\ge n_0$, and so $f(X)= \sum_{n \ge n_0, p|n} a_nX^n = F(X^p)$, where
$
F(X):= \sum_{l \ge n_0/p}a_{lp}X^{lp},
$
 a rational function since the coefficients $\{a_{lp}\}_{l\ge n_0/p}$ satisfy a linear recurrence.
\end{proof}

\begin{lemma} \label{DE} Let $f,g$ be rational functions over $\mathbb F_p$ such that $f'(X)=X^{-1}g(X^p)$. Then $f'(X)=0$. 
\end{lemma}

\begin{proof} Suppose that $f'(X)=X^{-1}g(X^p)$ and consider the Laurent expansions of both sides.  Every monomial appearing in $X^{-1}g(X^p)$ has an exponent $e \equiv -1$ mod $p$, but all such monomials in $f'(X)$ have zero as a coefficient. Therefore, $f'(X)=0$.
\end{proof}

\begin{proof}[Proof of Proposition \ref{tprop}] The proof is by induction on $k$.
 The case $k=1$ is immediate from Lemmas \ref{Fieldlemma} and \ref{fprimefplemma}. Indeed, if $p|(f'+g'/g)$ then by the first lemma, $p|f'$ and $p|g'$, and by the second, $f(X)\equiv f_0(X^p)$ mod $p$, $g(X) \equiv g_0(X^p)$ mod $p$, for some rational functions $f_0,g_0$. Thus the proposition holds with $u(X)=0$. 
 
Suppose that the proposition, and consequently Corollary \ref{fspcorollary},  hold for $k$, and assume that $f,g$ are such that $p^{k+1}|(f'+g'/g)$. In particular $f$ and $g$ are as given in \eqref{fsp} and \eqref{gsp}, so
\begin{align}
f(X) & = \sum_{i=0}^{k-1} p^if_i(X^{p^{k-i}}) -\log(1+pu(X)) -p^kf_k(X), \label{fsp*}\\
g(X) &= g_0(X^{p^k})(1+pu(X))(1+p^kg_k(X)),\label{gsp*}
\end{align} 
for some rational $f_k(X),g_k(X)$.  Throughout the proof, we view $\log(1+pX)$ as a polynomial in $X$ mod $p^{k+1}$.
 Then
\begin{align*}
0 &\equiv f'(X)+\frac {g'(X)}{g(X)} \equiv p^k\sum_{i=0}^{k-1} X^{p^{k-i}-1}f_i'(X^{p^{k-i}})-p^kf_k'(X)\\
& \quad \quad +p^kX^{p^k-1}\frac {g_0'(X^{p^k})}{g_0(X^{p^k})} +p^k\frac {g_k'(X) }{1+p^kg_k(X)} \pmod{p^{k+1}},
\end{align*}
 yielding
 \begin{align*}
0 \equiv \sum_{i=0}^{k-1} X^{p^{k-i}-1}f_i'(X^{p^{k-i}}) +X^{p^k-1}\frac {g_0'(X^{p^k})}{g_0(X^{p^k})} +g_k'(X)-f_k'(X) \pmod{p},
\end{align*}
 that is,
 $$
 f_k'(X)-g_k'(X) \equiv \sum_{i=1}^{k-1} X^{p^{k-i}-1}f_i'(X^{p^{k-i}}) +X^{p^k-1}\Big(f_0'(X^{p^k})+\frac {g_0'(X^{p^k})}{g_0(X^{p^k})}\Big) \pmod p.
 $$
 By Lemma \ref{DE}, it follows that both sides are zero mod $p$. Applying Lemma \ref{fprimefplemma} to the left-hand side, we have  $f_k(X) =g_{k}(X)-f_{k,0}(X^p)-pf_{k,1}(X)$ for some $f_{k,0},f_{k,1}$.  Dividing the right-hand side by $X^{p-1}$ and putting $Y=X^p$, we get
 \begin{equation*}
 0 \equiv  \sum_{i=1}^{k-1} Y^{p^{k-1-i}-1}f_i'(Y^{p^{k-1-i}}) +Y^{p^{k-1}-1}\Big(f_0'(Y^{p^{k-1}})+\frac {g_0'(Y^{p^{k-1}})}{g_0(Y^{p^{k-1}})}\Big) \pmod p.
 \end{equation*}  
Applying Lemma \ref{DE} again, we see that $f_{k-1}'(Y)\equiv 0$ mod $p$, implying that $f_{k-1}(Y)= f_{k-1,0}(Y^p)+pf_{k-1,1}(Y)$ for some $f_{k-1,0}, f_{k-1,1}$. Dividing the preceding congruence by $Y^{p-1}$, and putting $Z=Y^p$, gives
   $$
   0\equiv  \sum_{i=1}^{k-2} Z^{p^{k-i-2}-1}f_i'(Z^{p^{k-i-2}}) +Z^{p^{k-2}-1}\Big(f_0'(Z^{p^{k-2}})+\frac {g_0'(Z^{p^{k-2}})}{g_0(Z^{p^{k-2}})}\Big) \pmod p.
   $$
    Continuing in this manner $k$ times, we find that
   for $1 \le i \le k-1$, 
   $$
   f_i(X)=f_{i,0}(X^p) + pf_{i,1}(X),
   $$
   for some $f_{i,0}$, $f_{i,1}$. On the final iteration we are left with
   $$
   f_0'(X) +\frac {g_0'(X)}{g(X)} \equiv 0 \pmod p,
   $$
   implying (by the $k=1$ case) that
   \begin{align*}
   f_0(X)&=  -f_{0,0}(X^p)-pf_{0,1}(X),\\
   g_0(X) &= g_{0,0}(X^p)+pg_{0,1}(X),
   \end{align*}
   for some $f_{0,0},f_{0,1},g_{0,0},g_{0,1}$. 
   
Returning to \eqref{fsp*} and \eqref{gsp*} we find that
\begin{align*}
f(X)  
 &= \sum_{i=0}^{k-1} p^i\big(f_{i,0}(X^{p^{k+1-i}})+pf_{i,1}(X^{p^{k-i}})\big) -\log(1+pu(X)) 
 -p^kf_k(X)\\
 &= \sum_{i=0}^{k} p^i\big(f_{i,0}(X^{p^{k+1-i}})+pf_{i,1}(X^{p^{k-i}})\big) -\log(1+pu(X)) 
 -p^kg_k(X)\\
g(X) &= \big(g_{0,0}(X^{p^{k+1}})+pg_{0,1}(X^{p^k})\big)(1+pu(X))(1+p^kg_k(X))\\
& = g_{0,0}(X^{p^{k+1}}) \Big(1+pv(X^{p^k})\Big)(1+pu(X))(1+p^kg_k(X)),
\end{align*} 
where $v(Y)=g_{0,1}(Y)/g_{0,0}(Y^p)$. 
 Define $U(X)$ by
$$
(1+pU(X))= (1+pv(X^{p^k}))(1+pu(X))(1+p^kg_k(X)),
$$
so that
$$
g(X)= g_{0,0}(X^{p^{k+1}})(1+pU(X)),
$$
$$
-\log(1+pu(X))\equiv  -\log(1+pU(X))+p^kg_k(X) +pV(X^{p^k}) \pmod {p^{k+1}},
$$
with  $V(Y)=p^{-1}\log(1+pv(Y))$, and
$$
f(X)\equiv  \sum_{i=0}^{k+1} p^iF_i(X^{p^{k+1-i}}) - \log(1+pU(X)) \pmod {p^{k+1}},
$$
where $F_0(X)= f_{0,0}(X)$, $F_1(X)=f_{0,1}(X)+f_{1,0}(X)+V(X)$, $\dots$, 
$
F_k(X)=f_{k,0}(X^p)+f_{k-1,1}(X^p)+W(X),
$
and $F_{k+1}(X)= f_{k,1}(X)$, completing the proof.

\end{proof}

\section{Conversion of a mixed sum to a pure exponential sum} \label{conversion} 
In this section we  take the first steps towards proving Theorem \ref{maintheorem}.  Let $f,g$ be rational functions over $\mathbb
Z$, not both constants, $p$  an odd prime with $\text{ord}_p(f) \ge 0$ and $\text{ord}_p(g) =0$, $m$ a positive integer with  $m \ge t +2$ and $\chi$ a multiplicative
character mod ${p^m}$. 
  
For any positive integer $u$, with $g(u)$ defined mod $p$ and $p \nmid g(u)$,  let $F_u(Y)$  be the formal power series in $Y$ defined by
\begin{equation} \label{F}
F_u (Y) := c_\chi \log \left( \frac {g(u
+pY)}{g(u)}\right)
+ f(u +pY) -f(u),
\end{equation} 
so that
\begin{equation} \label{Fprime}
F_u'(Y)= p^{t+1} \mathcal C(u+pY),
\end{equation}
where $\mathcal C$ is the critical point function \eqref{cpfunction}.
Write $x=u+p^{m-t-1}v$ with $u$ running from 1 to $p^{m-t-1}$ and
$v$ running from 1 to $p^{t+1}$. Then with $m_1:=m-t-1$, 
\begin{align*}
&S(\chi,g,f,p^m)=\sum_{u=1}^{p^{m_1}}\  \sum_{v=1}^{p^{t+1}} \chi(g(u+p^{m-t-1}v))e_{p^m}(f(u+p^{m-t-1}v))\\
& =\sum_{u=1}^{p^{m_1}} \chi(g(u))e_{p^m}(f(u)) \sum_{v=1}^{p^{t+1}} \chi\left(\frac {g(u+p^{m-t-1}v)}{g(u)}\right)
e_{p^m}\Big(f(u+p^{m-t-1}v)-f(u)\Big)\\
&= \sum_{u=1}^{p^{m_1}} \chi(g(u))e_{p^m}(f(u)) \sum_{v=1}^{p^{t+1}} e_{p^m}\left(c_\chi\log\Big(\frac {g(u+p^{m-t-1}v)}{g(u)}\Big) +f(u+p^{m-t-1}v)-f(u)\right), 
\end{align*}  
the final equality following from  \eqref{plog}.
Thus by the definition of $F_u$ we have
 \begin{equation} \label{Schiconvert}
S(\chi,g,f,p^m)=\sum_{u=1}^{p^{m-t-1}} \chi(g(u))e_{p^m}(f(u)) \sum_{v=1}^{p^{t+1}} e_{p^m}\Big(F_u(p^{m-t-2}v)\Big) .
\end{equation}
Now, by \eqref{Fprime}, $p^{t+1}|F_u'(Y)$, and $p^{t+2}|F_u^{(k)}(Y)$ for $k \ge 2$, and so for $m \ge t+2$, it follows by a Taylor series expansion, that
$$
F_u(p^{m-t-2}v) \equiv F_u(0)+F_u'(0)p^{m-t-2}v  \pmod {p^m}.
$$
Since $F_u(0)=0$ and $F_u'(0) =p^{t+1}\mathcal C(u)$, we get
$$
F_u(p^{m-t-2}v) \equiv p^{m-1} \mathcal C(u)\, v \pmod {p^m}.
$$
Thus, the sum over $v$ in \eqref{Schiconvert} vanishes unless $p|\mathcal C(u)$, that is, $u\equiv \alpha$ mod $p$, for some critical point $\alpha$. In the latter case, the sum over $v$ is equal to $p^{t+1}$. In particular, letting $\mathcal A$ be a set of integer representatives for the set  of critical points \eqref{Adef}, we see that 
\begin{equation} \label{Schi1}
S(\chi,g,f,p^m)=\sum_{\alpha \in \mathcal A} S_\alpha(\chi,g,f,p^m) = p^{t+1} \sum_{\alpha \in \mathcal A}\  \sum_{u\equiv \alpha \ \text{mod}\ p}^{p^{m-t-1}} \chi(g(u))e_{p^m}(f(u)).
\end{equation}

Suppose now that $\alpha\in \mathbb Z$ is a fixed critical point and consider computing $S_\alpha(\chi,g,f,p^m)$.
Write $u=\alpha+py$ with $y$ running from 1 to $p^{m-t-2}$ to get
\begin{align}
S_\alpha(\chi,g,f,p^m)&= p^{t+1} \sum_{y=1}^{p^{m-t-2}}\chi (g(\alpha+py)) e_{p^m}(f(\alpha +py))\notag \\
&=  p^{t+1}\chi(g(\alpha)) e_{p^m}(f(\alpha))
\sum_{y=1}^{p^{m-t-2}} e_{p^m} (F_\alpha(y)). \label{SalphaF}
\end{align}   
Expand
$F_\alpha(Y)$ into a formal power series 
\begin{equation} \label {aj}
F_\alpha (Y)= \sum_{j=1}^\infty a_j Y^j,
\end{equation}
with $p$-adic integer coefficients $a_j$.
Define
\begin{equation}
 \label{sigma}
\sigma:= \text{ord}_p(F_\alpha(Y))= \min_{j\ge 1} \{\text{ord}_p(a_j)\},
\end{equation}
and
\begin{equation} \label{galph}
G_\alpha(Y):=
p^{-\sigma}F_\alpha (Y).
\end{equation}
Then by   \eqref{SalphaF} we have the following conversion of $S_\alpha (\chi,g,f,p^m)$ into a pure exponential sum.
  
\begin{proposition} \label{propconvert} Suppose that $f,g$ are rational functions over $\mathbb Q$, not both constants,  $p$ is an odd prime with $\text{ord}_p(f) \ge 0$ and $\text{ord}_p(g) = 0$, and $m \ge t+2$.  If $\alpha$ is not a critical point for the sum $S(\chi,g,f,p^m)$, then $S_\alpha=0$. If $\alpha$ is a critical point,  then
\begin{equation} 
S_\alpha (\chi,g,f,p^m)= \begin{cases} p^{\sigma -1}\chi(g(\alpha)) e_{p^m}(f(\alpha))
S(G_\alpha,p^{m-\sigma}), & \text{if $m>\sigma$;}\label{sg1}\\ 
 p^{m-1}\chi(g(\alpha)) e_{p^m}(f(\alpha)), & \text{if $m \le \sigma$.}\end{cases}
\end{equation}
where $S(G_\alpha,p^{m-\sigma})= \sum_{y=1}^{p^{m-\sigma}}
e_{p^{m-\sigma}}(G_\alpha (y))$. 
\end{proposition}
\noindent

The function $G_\alpha$, defined apriori as an infinite series with $p$-adic
coefficients, may be viewed
as a polynomial over $\mathbb Z$ in the exponential sum $S(G_\alpha,
p^{m-\sigma})$,
since its coefficients are $p$-adic integers and
the high order coefficients all vanish modulo $p^{m-\sigma}$; see \eqref{galph2}. Thus
$S(G_\alpha, p^{m-\sigma})$ is just an ordinary pure exponential sum.

\section{Pure exponential sum bounds}

 We estimate $|S(G_\alpha, p^{m-\sigma})|$ using  the bound for pure exponential sums given in the following theorem. \.
Let $f$ be a nonconstant  polynomial mod $p$ over $\mathbb Z$, and $t=t(f,p)$ be defined by $p^t\|f'(X)$.
Let $\lambda=p^{\frac 2{p+1}}$, $\beta= p^{\frac 1{\sqrt{p}+1}}$ and for $p \ge 17$, let $\{\beta_j\}$ be the  sequence
\begin{equation} \label{bjdef}
\beta_j:= \begin{cases} \beta, & \text{if $1 \le j <\sqrt{p}$};\\
p^{\frac 1{j+1}}, & \text{if $\sqrt{p}<j\le \tfrac {p-3}2$};\\
\lambda, & \text{if $j \ge \tfrac {p-1}2$}.
\end{cases}
\end{equation}

 \begin{theorem} \label{mainthm1} \cite[Theorem 3.1]{cochrane2024} Let $p$ be a prime, $m,d_1$ positive integers, $f(x)$ a polynomial with $\deg_p(p^{-t}f')=d_1$ and $m-t\ge 2$ for $p$ odd, $m-t \ge 3$ for $p=2$. Then
\begin{equation} \label{thm1e}
 |S(f,p^m)| \le \begin{cases} \lambda\, p^{\frac t{d_1+1}}p^{m(1-\frac 1{d_1+1})}, & \text{for $p \le 13$;}\\ 
 \beta_{d_1}\, p^{\frac t{d_1+1}}p^{m(1-\frac 1{d_1+1})}, & \text{for $p \ge 17$.}
 \end{cases}
 \end{equation}  
\end{theorem}
\noindent In particular, since $p^{\frac 2{p+1}} \le \sqrt{3}$ and $p^{\frac 1{\sqrt{p}+1}} \le 13^{\frac {1}{\sqrt{13}+1}}\le 1.75$ for any prime $p$,   we have uniformly under the hypotheses of the theorem,
\begin{equation} \label{pure2}
 |S(f,p^m)| \le 1.75\,  p^{\frac t{d_1+1}}p^{m(1-\frac 1{d_1+1})}.
 \end{equation}
 See also \cite{czpuremix}, \cite[Theorem 2.1]{cz3}, \cite{ding2}, \cite{loh} and
\cite{lv} for related bounds. 
For critical points $\alpha$ of multiplicity one, we have the following more precise result.

\begin{theorem} \label{cp1theorem} \cite[Theorem 2.1]{czpuremix}. Let $\alpha$ be a critical point of multiplicity one for the sum $S(f,p^m)$. 
If $p$ is odd  and $m \ge t+2$, or $p=2$ and $m \ge t+3$ then $|S_\alpha|=p^{\frac {m+t}2}$.  
  If $p$ is odd we have in fact
 $$
S_\alpha (f,p^m)
 =\begin{cases}   e_{p^m}(f(\alpha^*))p^{\frac {m+t}2},
 \quad &\text{if $m-t$ is even;}\\
e_{p^m}(f(\alpha^*)){\left(\tfrac
{A_\alpha}p\right)} \Cal G_p p^{\frac
{m+t-1}2},
\quad &\text{if $m-t$ is odd;} \end{cases}
$$
where $\alpha^*$ is the unique lifting of $\alpha$ to a solution of the
congruence

\noindent
 $p^{-t}f'(x)\equiv 0 \pmod {p^{\lfloor \frac {m+t}2\rfloor}}$, and
$
A_\alpha \equiv 2p^{-t}f''(\alpha) \pmod p .
$
Here $\Cal G_p$ is the quadratic Gauss sum,
\begin{equation} \label{E:1.12}
\Cal G_p := \sum_{x=0}^{p-1} e_p(x^2) = \sum_{x=1}^{p-1} \left(\frac xp\right)
 e_p(x) =
\begin{cases} \sqrt{p} \quad \text{\ if $p
\equiv 1 \pmod 4$,} \\ i\sqrt{p} \quad \text{if $p \equiv 3 \pmod 4$},
\end{cases}
\end{equation}
and ${\left(\tfrac {A_\alpha}p\right)}$ is the Legendre symbol.  

\end{theorem}

\section{Relations between parameters}

Suppose that $\alpha \in \mathbb F_p$ is a critical point of 
multiplicity $\nu \ge 1$, that is, a zero  of \eqref{cpc}
of multiplicity $\nu$.
Develop $\mathcal C$ into a  Taylor series expansion about $\alpha$, say
\begin{equation} \label{cj}
\mathcal C(X) = \sum_{j=0}^\infty c_j(X-\alpha)^j,
\end{equation}
with $p$-adic integer coefficients $c_j$.  Since $\alpha$ is a zero of $\mathcal C$ mod $p$ of multiplicity $\nu$,
\begin{equation}
\text{ord}_p(c_j)>0  \quad \text{for $0\le j <\nu$\quad and \quad }
\text{ord}_p(c_\nu)=0.
\end{equation}
By \eqref{Fprime} we have $F_\alpha'(Y)= p^{t+1}\mathcal C(\alpha+pY)= p^{t+1} \sum_{j=0}^\infty c_j(pY)^j$ and so
\begin{equation} \label{Falph2}
F_\alpha (Y) = p^{t+1} \sum_{j=0}^\infty c_j p^j \frac {Y^{j+1}}{j+1},
\end{equation}
and 
\begin{equation} \label{galph2}
G_\alpha(Y)=
p^{-\sigma}F_\alpha (Y)= p^{-\sigma}\sum_{j=1}^{\infty} a_jY^j
=p^{t-\sigma}\sum_{j=1}^{\infty} \frac {c_{j-1}}j p^jY^j.
\end{equation}

Set
\begin{align}
\label{tau}
&\tau:= \text{ord}_p(G_\alpha'(Y)),\\
\label{H}
&H_\alpha (Y) :=
p^{-\tau}G_\alpha'(Y)=p^{-\tau-\sigma}\sum_{j=1}^{\infty} a_jjY^{j-1}=
p^{t-\tau-\sigma}\sum_{j=1}^{\infty} c_{j-1}p^jY^{j-1},
\end{align}
noting that $H_\alpha(Y)$ has $p$-adic integer coefficients. We readily
obtain the following relationships:
\begin{align}
\label{sigmalow}
&\sigma \ge t+2,\\
\label{sigmaup}
&\sigma \le \nu+1+t-\tau, \\
\label{dpg2}
&\deg_p(G_\alpha) \le \sigma -t +\text{ord}_p(\deg_p(G_\alpha)),\\
\label{dph}
& \deg_p(H_\alpha) \le \sigma +\tau-t-1\le \nu,\\
\label{tauup}
&\tau \le \text{ord}_p(\deg_p(G_\alpha)).
\end{align}
\noindent
The first inequality \eqref{sigmalow} follows from \eqref{Falph2} and the fact
$p|c_0$. Inequalities \eqref{sigmaup} and \eqref{dph}
 follow from the second series
expansion of $H_\alpha$ in \eqref{H}, setting $j=\nu+1$ and
$j=\deg_p(H_\alpha)+1$ respectively, while  inequality
\eqref{dpg2} follows from the second series expansions of $G_\alpha$ in
\eqref{galph2}, setting
$j=\deg_p(G_\alpha)$. Finally, to obtain \eqref{tauup}, set $j=\deg_p(G_\alpha)$
and note that by definition $\text{ord}_p(a_j)=\sigma$. Then, since $p^\tau| G_\alpha'(Y)$, we have
$$
\tau \le \text{ord}_p(a_jj)-\sigma = \text{ord}_p(j).
$$

\section{Proof of Theorem \ref{maintheorem}}

Let $\sigma, \tau$ and $t$ be as defined in \eqref{sigma}, \eqref{tau} and \eqref{t}.
Suppose that $m \ge t+2$. We already saw (Proposition \ref{propconvert})
 that if $\alpha$ is not a critical point, then $S_\alpha=0$.
Suppose now 
that $\alpha$ is a critical point of
multiplicity $\nu \ge 1$.  We prove parts $(ii)$ and $(iii)$ of the theorem simultaneously by considering different cases.  If $\nu=1$ then from \eqref{sigmalow} and \eqref{sigmaup}  we have $t+2 \le \sigma \le t+2-\tau$ and so $\tau=0$ and $\sigma=t+2$. For such critical points we are able to evaluate $S_\alpha$ explicitly as indicated in the proof.

Case $i$. Suppose first that $\sigma \ge m$.
 Then by Proposition \ref{propconvert}, 
$$
|S_\alpha|= p^{m-1} = p^{\frac {m-\nu-1}{\nu+1}}p^{m(1-\frac 1{\nu+1})} \le 
p^{\frac {\sigma-\nu-1}{\nu+1}}p^{m(1-\frac 1{\nu+1})}\le 
p^{\frac {t}{\nu+1}}p^{m(1-\frac 1{\nu+1})},
$$  
the last inequality following from \eqref{sigmaup}. 

If in addition we have $\nu=1$, so that $\sigma=t+2$, then, $t+2 \le m \le \sigma=t+2$ and so $m=t+2$. By Proposition \ref{propconvert},  it follows that
\begin{equation} \label{nu1casei}
S_\alpha = p^{m-1}\chi(g(\alpha))e_{p^m}(f(\alpha)) = p^{\frac {m+t}2}\chi(g(\alpha))e_{p^m}(f(\alpha)).
\end{equation}

Case $ii.$  Suppose next that $\sigma =m-1$.
We have trivially
$$
|S_\alpha|\le p^{m-1}\le 1.75\,  p^{\frac t{\nu+1}}p^{m(1-\frac 1{\nu+1})},
$$
provided that $p^{\frac {m-t}{\nu+1}-1} \le 1.75$, that is $p^{\frac {\sigma-t-\nu}{\nu+1}} \le 1.75$.  By \eqref{sigmaup}, it suffices to have $p^{1-\tau} \le 1.75^{\nu+1}$, which is the case if $\tau \ge 1$ or $\tau=0$ and $p \le 1.75^{\nu+1}$. Suppose now that
 $\tau =0$ and $p>1.75^{\nu+1}$.
Let
$d_p=\deg_p(G_\alpha) \ge 1$.
By \eqref{dpg2} and \eqref{sigmaup} we have
\begin{equation} \label{E:3.3}
d_p \le \nu +1 +\text{ord}_p(d_p).
\end{equation}
Suppose that $\text{ord}_p(d_p) \ge 1$.  If $d_p=p$ then by \eqref{E:3.3}
$p\le
\nu+2$, contradicting our assumption that $p >1.75^{\nu+1}$.
Otherwise $d_p \ge 2p$ and thus since $\text{ord}_p(d_p) \le d_p/2$ we have by
\eqref{E:3.3} that
$$
p \le \tfrac 12 d_p \le d_p - \text{ord}_p(d_p) \le \nu +1,
$$
again contradicting  $p >1.75^{\nu+1}$.
Thus we must have $\text{ord}_p(d_p) = 0$ and so by \eqref{E:3.3}, $d_p \le
\nu+1$. It follows from Proposition \ref{propconvert} and 
the mod $p$ upper bound \eqref{weilub2}  that
\begin{align}
|S_\alpha| &= p^{\sigma -1}|S(G_\alpha,p)|\le 1.75\ p^{\sigma -\frac 1{d_p}}
\le 1.75\, p^{m-1-\frac 1{\nu+1}} \notag \\
&= 1.75\, p^{\frac t{\nu+1} }p^{m(1-\frac 1{\nu+1})}p^{\frac {\sigma
-\nu-1-t}{\nu+1}}. \notag
\end{align}
The bound \eqref{mainlocalub} now follows from \eqref{sigmaup}.

If $\nu=1$, we can be more precise. The prime $p=3$ will need special attention because of the third order term with denominator 3 in the expansion of $G_\alpha$. We deal with it in the next paragraph.  As noted above, $\tau=0$ and $\sigma=t+2$. We  get from \eqref{galph2} that for $p>3$, $d_p(G_\alpha)=2$, and so $S(G_\alpha,p)$ is just a quadratic Gauss sum.  Thus
\begin{align*}
|S_\alpha| &= p^{\sigma -1}|S(G_\alpha,p)| = p^{\sigma-1} \sqrt{p}= p^{\frac t2}p^{\frac m2}.
\end{align*}
Let $\alpha^*$ be a lifting of $\alpha$ to a solution of $\mathcal C(x) \equiv 0$ mod $p^2$ so that $c_0\equiv 0$ mod $p^2$.  Then in the expansion of $G_\alpha(Y)$ in \eqref{galph2} we have for $p>3$, 
$
G_\alpha(Y) \equiv  \overline 2 c_1 Y^2\pmod p,
$
where $c_1=\mathcal C'(\alpha^*)$.  Evaluating the Gauss sum we get
$S(G_\alpha,p)= \chi_2(2c_1) \mathcal G_p$, and thus by Proposition \ref{propconvert}, $S_\alpha=p^{\frac {m+t-1}2}\chi(g(\alpha^*))e_{p^m}(f(\alpha^*))\chi_2(2c_1) \mathcal G_p$, and $|S_\alpha|=p^{\frac {m+t}2}$.

If $p=3$, then $G_\alpha(Y)=\overline 2 c_1Y^2 +c_2Y^3$. Noting that $y^3 \equiv y$ mod $3$ for all $y \in \mathbb Z$, we have $S(G_\alpha,3)= S(c_2Y+\overline 2 c_1Y^2,3)$ and so again one can write down an explicit evaluation in terms of a quadratic Gauss sum.  For our purposes we will simply note that $|S(G_\alpha,3)| =\sqrt{3}$, which again yields $|S_\alpha|=p^{\frac {m+t}2}$.

Case $iii$. Suppose that $m-1-\tau \le \sigma \le m-2$.
In particular, we must have $\tau \ge 1$ and consequently $\nu \ge 2$ as noted in the opening paragraph.  Then we have the trivial
estimate,
$$
|S_\alpha|\le p^{m-1}=p^{\frac {m-\nu-1}{\nu+1}}p^{m(1-\frac 1{\nu+1})}
$$
\begin{equation} \label{E:3.4}
\le p^{\frac 1{\nu+1}} p^{\frac {\sigma +\tau -\nu-1}
{\nu+1}}p^{m(1-\frac 1{\nu+1})}
\le 
p^{\frac 1{\nu+1}} p^{\frac t
{\nu+1}}p^{m(1-\frac 1{\nu+1})},
\end{equation}
the latter inequality following from \eqref{sigmaup}. Again, put $d_p=\deg_p(G_\alpha) \ge 1$. Since $\tau \ge 1$ we get from \eqref{tauup} that $\text{ord}_p(d_p) \ge 1$.
By \eqref{dpg2} and \eqref{sigmaup}, $d_p \le \nu+1-\tau+ \text{ord}_p(d_p)$,
and so
$$
p-1 \le p^{\text{ord}_p(d_p)}-\text{ord}_p(d_p)\le d_p-\text{ord}_p(d_p)\le \nu
+1-\tau \le \nu.
$$
Thus $p^{\frac 1{\nu+1}} \le (\nu+1)^{\frac 1{\nu+1}} \le 3^{1/3}<1.75$
and so \eqref{mainlocalub}
follows from \eqref{E:3.4}.

Case $iv$. Suppose finally that $\sigma \le m-2-\tau$, that is, $m-\sigma \ge \tau+2$. 
In this case we  can apply Theorem \ref{mainthm1}  to
$S(G_\alpha,p^{m-\sigma})$ and
obtain from \eqref{pure2} with $d_1:=\deg_p(H_\alpha)=\deg_p(p^{-\tau}G_\alpha')$, 
$$
|S_\alpha| =p^{\sigma -1} |S(G_\alpha,p^{m-\sigma})| \le
1.75\, p^{\sigma
-1} p^{\frac {\tau}{d_1+1}}p^{(m-\sigma)(1-\frac
1{d_1+1})}.
$$
Now  by \eqref{dph},
$d_1=\deg_p(H_\alpha)\le \nu$ and thus since $m-\sigma -\tau >0$ we obtain
$$
|S_\alpha|\le 1.75\, p^{\sigma-1}p^{\frac
{\tau}{\nu+1}}p^{(m-\sigma)(1-\frac
1{\nu+1})}\le 1.75\, 
 p^{\frac {\tau+\sigma -\nu-1}{\nu+1}} p^{m(1-\frac
1{\nu+1})}.
$$ 
Inequality \eqref{mainlocalub} follows from \eqref{sigmaup}.

If $\nu=1$, then as noted in the opening paragraph,  $\sigma=t+2$ and $\tau=0$. Thus $m \ge \sigma+2 =t+4$ and the first order term of $H_\alpha(Y)$ is $p^{t-\tau-\sigma}c_1p^2 Y= c_1Y$, with $p \nmid c_1$,  so $d_p(H_\alpha)=1$. There is a single critical point $\alpha$ of multiplicity one for the sum $S(G_\alpha, p^{m-\sigma})$.  If $\alpha^*$ is a lifting of $\alpha$ to a solution of $\mathcal C(x) \equiv 0$ mod $p^2$, so that $c_0\equiv 0$ mod $p^2$, then   $H_\alpha(Y) \equiv  c_1 Y$ mod $p$, with $p \nmid c_1$, and
the critical point is just $y=0$. 
 Thus for $m \ge t+4$, noting that $G_\alpha(0)=0$, we get from Theorem \ref{cp1theorem},
$$ 
S(G_\alpha,p^{m-\sigma})= \begin{cases} p^{\frac {m-\sigma}2}, & \text{if $m-\sigma$ is even;}\\
 \chi_2(A_\alpha)\mathcal G_p p^{\frac {m-\sigma-1}2}, & \text{if $m-\sigma$ is odd;} \end{cases}
$$  
where $A_\alpha \equiv 2  G_\alpha''(0)\equiv 2 c_1 \equiv  2\mathcal C'(\alpha^*)$ mod $p$. It follows from Proposition \ref{propconvert} that 
$$
S_\alpha(\chi,g,f,p^m) =\begin{cases} \chi(g(\alpha^*)) e_{p^m}(f(\alpha^*)) p^{\frac {m+t}2}, & \text{if $m-\sigma$ is even;}\\
 \chi(g(\alpha^*)) e_{p^m}(f(\alpha^*)) \chi_2(A_\alpha)\mathcal G_p p^{\frac {m+t-1}2}, & \text{if $m-\sigma$ is odd.} \end{cases}
$$
 
 \begin{remark} \label{errorremark2} The special consideration needed when $p=3$ and $\sigma=m-1$, case $(ii)$ above, was overlooked in the formula of \cite[Theorem 1.1]{cochrane2002}. The formula should be modified as indicated in the proof above.
 \end{remark}
 
\section{Proof of Corollary \ref{maincor1} }
Put $d=\deg_p(\mathcal C_+)$, $S=S(\chi,g,f,p^m)$. Assume that $m \ge t+2$. 
  For $1 \le j \le d-1$, let $n_j$ denote the number of critical points of multiplicity $j$ and put $x_j:=jn_j$, so that $\sum_{j=1}^{d}x_j\le d$. Put $\delta := p^{t-m}$. 
Then by Theorem \ref{maintheorem} $(iii)$
we have 
\begin{align*}
|S(\chi,g,f,p^m)| \le1.75\, \sum_{j=1}^{d} n_j  p^{\frac t{j+1}}p^{m(1-\frac 1{j+1})} = 1.75\, p^m \sum_{j=1}^{d} x_j \delta^{\frac 1{j+1}}/j.
\end{align*}
View the latter sum as a linear function in the $x_j$ subject to the constraint $\sum_{j=1}^d x_j \le d$.  The maximum value of the linear function over  the simplex occurs at one of the vertex points having one coordinate equal to $d$ and the others all equal to zero, and so $|S| \le 1.75 d p^m \max_{1 \le j \le d} \delta^{\frac 1{j+1}}/j$.
 Now the sequence $\{\delta^{\frac 1{j+1}}/j\}_{j=1}^d$ is increasing provided that $\delta  \le (1+\frac 1{d-1})^{-d(d+1)}$ as one can verify by computing the ratio of successive terms (see \cite[Lemma 12.1]{cochrane2024}).  For such $\delta$ we conclude that
$$
|S(\chi,g,f,p^m)| \le 1.75\, p^m \delta^{\frac 1{d+1}} = 1.75\, p^{\frac t{d+1}}p^{m(1-\frac 1{d+1})},
$$
as desired.  For $\delta > (1+\frac 1{d-1})^{-d(d+1)}$ we have $p^{\frac {m-t}{d+1}} < (1+\frac 1{d-1})^d <3$ for $d \ge 6$, and so  
$$
|S(\chi,g,f,p^m)| \le p^m \le  3 \, p^{\frac t{d+1}}p^{m(1-\frac 1{d+1})}.
$$
 
We are left with the case where $1 \le d \le 5$ .  If there is a single critical point $\alpha$ of multiplicity $d$, then by Theorem \ref{maintheorem} $(iii)$,
$
|S|=|S_\alpha|\le 1.75\, p^{\frac t{d+1}} p^{m(1-\frac 1{d+1})},
$
as desired. If all critical points have multiplicity one, then by part $(ii)$ of Theorem \ref{maintheorem},
$
|S| \le d p^{\frac {m+t}2} \le 3\, p^{\frac t{d+1}} p^{m(1-\frac 1{d+1})},
$
provided that $\frac d3 \le p^{(m-t) \frac {d-1}{2d+2}}$. Since $m-t \ge 2$ it suffices to have $\frac d3 \le p^{\frac {d-1}{d+1}}$, which is the case for $p \ge 3$ and $d \le 5$. The corollary is now established for $d=1$ and $d=2$. We turn to the remaining cases for $d=3,4$ and 5.

Suppose that $d=3$ and that there is
 one critical point of multiplicity 2 and one of multiplicity 1. Then $|S| \le 1.75\, p^{\frac t3}p^{\frac 23 m}+ p^{\frac t2 +\frac m2}<2.75\,  p^{\frac t4}p^{\frac 34 m}$, as desired since $\frac t3+\frac 23 m \le \frac t4 +\frac 34 m$ and $\frac t2+\frac m2 \le \frac t4 +\frac 34 m$ for $m \ge t$.
 
   For $d=4$, consider the case of two critical points of multiplicity 2, where $|S| \le 2\cdot 1.75\, p^{\frac t3} p^{m(1-\frac 13)} \le 3\, p^{\frac t5} p^{m(1-\frac 15)}$, since $p^{m-t} \ge p^2 \ge (3.5/3)^{\frac {15}2}$; next if there is one critical point of multiplicity 1 and one of multiplicity 3, then $|S| \le p^{\frac {m+t}2} +1.75\, p^{\frac t4+\frac 34 m}\le 2.75\, p^{\frac t5+\frac 45 m}$; finally if there is one critical point of multiplicity 2 and two of multiplicity 1, then $|S| \le 2p^{\frac{m+t}2} + 1.75\, p^{\frac t3+\frac 23 m} \le  3 p^{\frac t5+\frac 45 m}$, since $m-t \ge 2$ and $\frac 23 p^{-\frac 35}+\frac {1.75}3 p^{-\frac 4{15}}<1$.

   Suppose finally that $d=5$. For the case of one critical point of multiplicity 4 and one of multiplicity 1 we have $|S|\le p^{\frac {m+t}2}+ 1.75 \, p^{\frac t5+ \frac 45m} <2.75\, p^{\frac t6}p^{\frac 56 m}$;
 for one point of multiplicity 3 and one of multiplicity 2, $|S| \le 1.75\, p^{\frac t4+\frac 34 m}+1.75\, p^{\frac t3+\frac 23 m}< 3\, p^{\frac t6}p^{\frac 56 m}$, since $\frac {1.75}{3}(p^{-\frac 16}+p^{-\frac 13})<1$; for two points of multiplicity 2 and one of multiplicity 1, we have $|S| \le  p^{\frac t2+\frac m2} + 3.5\ p^{\frac t3 +\frac 23 m}< 3 p^{\frac t6+\frac 56 m}$, for $p \ge 3$, since $\frac 13 p^{-\frac 23}+\frac {3.5}3 p^{-\frac 13}<1$; for one point of multiplicity 2 and three of multiplicity 1, we have $|S| \le 3p^{\frac {m+t}2}+1.75\, p^{\frac {t}3+\frac 23 m}< 3 p^{\frac t6+\frac 56 m}$, for $p \ge 3$, since $ p^{-\frac 23}+\frac {1.75}3 p^{-\frac 13}<1$ for $p \ge 3$.

\section{Proof of Corollary \ref{m2}}
Let $f,g$ be rational functions over $\mathbb Q$, not both constants, and  $\mathcal C= {\mathcal C_+}/\mathcal C_-$ denote the critical point function associated with $S(\chi,g,f,p^m)$.
Let $Z_1=\mathcal Z(f_-)$, the number of distinct zeros of $f_-$ in $\mathbb C$, $Z_2=\mathcal Z(g_+g_-)$, $Z_3$ the number of common zeros of $f_-$ and $g_+g_-$, and $D$ be as defined in the introduction,
\begin{equation} 
 D:= \deg(f) + \mathcal Z(f_-g_+g_-)= \deg(f)+ Z_1+Z_2-Z_3.
 \end{equation}

\begin{lemma}  \label{CDub} For any rational functions $f,g$ over $\mathbb Q$, not both constants,  prime power $p^m$, and character $\chi$ mod $p^m$,  we have
  \begin{align}\label{Cxdegub}   \deg(\mathcal C_+) &\le  D-1,\\
   \deg(\mathcal C_-) &\le \deg(f_-)+Z_1+Z_2-Z_3 \le D.
\end{align}
 In particular, the total number of critical points counted with multiplicity is at most $D-1$. 
 \end{lemma}
 
\begin{proof} 
  Say $f=f_+/f_-$ and $g=g_+/g_-$ have  factorizations over $\mathbb C$,
$$
f_+(X)= c_+ \prod_{i=1}^{Z_0}(X-\alpha_i)^{e_i}, \quad  f_-(X)=c_-\prod_{i=1}^{Z_1}(X-\beta_i)^{f_i},  \quad  
 g(X)=c_g\prod_{i=1}^{Z_2} (X-\gamma_i)^{g_i}
 $$
 for some constants $c_+,c_-,c_g$, positive integers $e_i,f_i$, and nonzero integers $g_i$.    Then
\begin{align*}
f'(X)&= \frac {f_+(X)}{f_-(X)}\Big( \sum_{i=1}^{Z_0}\frac {e_i}{X-\alpha_i} - \sum_{i=1}^{Z_1} \frac {f_i}{X-\beta_i}\Big)\\
&=\frac {\prod_{i=1}^{Z_1}(X-\beta_i)\Big(\sum_{i=1}^{Z_0} e_i f_+(X)/(X-\alpha_i) - f_+(X) \sum_{i=1}^{Z_1}f_i /(X-\beta_i)\Big)}{f_-(X) \prod_{i=1}^{Z_1} (X-\beta_i)}.
\end{align*}  
 The numerator of the latter expression is a polynomial of degree at most
$\deg(f_+)+Z_1-1$. Similarly,
$
\frac {g'(X)}{g(X)}=  \sum_{i=1}^{Z_2} \frac {g_i}{X-\gamma_i},
$
a rational function with denominator of  degree $Z_2$ and numerator of degree $Z_2-1$. Letting $Z_3$ denote the number of $\beta_i$ that are also zeros or poles of $g(X)$, we see that the least common denominator of $f'$ and $g'/g$ has degree at most $\deg(f_-)+Z_1+Z_2-Z_3$, and that the numerator of the sum $f'+c_\chi g'/g$ has degree at most
$$
\max\{\deg(f_+)+Z_1-1+Z_2-Z_3, \ \deg(f_-)+Z_1+Z_2-1-Z_3\}. 
$$
Thus 
$
\deg(\mathcal C_+) \le \deg(f)+Z_1+Z_2 -Z_3-1,
$
 as desired. 
 \end{proof}

 We turn now to the proof of Corollary \ref{m2}. Let $f,g$ be rational functions, not both constants.  For $m \le t+1$, we  have trivially
 $$
|S(\chi,g,f,p^m)| \le p^{m}=p^{\frac mD} p^{m(1-\frac 1{D})}\le p^{\frac {t+1}D}p^{m(1-\frac 1D)} .
$$
For $m\ge t+2$ the result is immediate from Corollary \ref{maincor1} and the bound $\deg_p(\mathcal C_+) \le D-1$ of Lemma \ref{CDub}.

\section{Proof of Theorem \ref{maintheorem0} and Corollary \ref{maincor} for odd $p$} \label{theorem0proofsec}
   
Let $f,g$ be rational functions over $\mathbb Z$,  not both constants (so that $D \ge 1$),  
$p$ an odd prime,
$m$ a positive integer and
$\chi$ a multiplicative character mod $p^m$ such that $S(\chi,g,f,p^m)$ is non-degenerate.
If $m=1$, then as noted earlier we have the stronger upper bound \eqref{weilub2}.

 ($i$) Suppose now that $m \ge 2$ and that
$\deg_p(f)\ge 1$.
  Then by  
Corollary \ref{tboundcor}, $p^t \le \deg_p(f)\le D$. If $2 \le m  \le t+1$ then $t \ge 1$,  $p^{t+1} \le p^{2t} \le D^2$ and so we have trivially
\begin{equation} \label{trivialtm}
|S(\chi,g,f,p^m)| \le p^{\frac mD} p^{m(1-\frac 1{D})}\le p^{\frac {t+1}D}  p^{m(1-\frac 1D)} \le D^{\frac 2D} p^{m(1-\frac 1D)}<2.1\, p^{m(1-\frac 1D)}.
\end{equation}
Finally, if $m \ge t+2$, then by Corollary \ref{m2}, 
\begin{align} \label{nontrivial}
|S(\chi,g,f,p^m)| &\le 3\, p^{\frac tD}p^{m(1-\frac 1D)} \le 3\, (\deg_p(f))^{\frac 1D} p^{m(1-\frac 1{D})} .
\end{align}

($ii$) Suppose that $m \ge 2$, $\chi$ is primitive ($t_\chi=0$) and $\deg_p(g) \ge 1$.  Then by Corollary \ref{tboundcor}, $p^{t} \le \deg_p(g)$.  Unfortunately, we cannot bound $\deg_p(g)$ in terms of $D$ since $D$ just depends on the number of zeros and poles of $g$, not their multiplicities.  
Thus proceeding as above, and using $p^{t+1} \le p^{2t}$,  \eqref{trivialtm} becomes
\begin{equation} \label{sim1}
|S(\chi,g,f,p^m)| \le p^{\frac {2t}D} p^{m(1-\frac 1D)} \le \deg_p(g)^{\frac 2D} p^{m(1-\frac 1D)},
\end{equation} 
and \eqref{nontrivial},
\begin{align} \label{sim2} 
|S(\chi,g,f,p^m)| &\le 3\, p^{\frac tD}p^{m(1-\frac 1D)} \le 3\, (\deg_p(g))^{\frac 1D} p^{m(1-\frac 1{D})} .
\end{align}

\begin{proof}[Proof of 
Corollary \ref{maincor}]  If $\deg_p(f) \ge 1$, then since $\deg_p(f) \le D$ and  $D^{\frac 1D} \le 3^{\frac 13}$ we have
$$
|S(\chi,g,f,p^m)| \le 3 \deg_p(f)^{\frac 1D} p^{m(1-\frac 1D)} \le 3^{\frac 43} p^{m(1-\frac 1D)}.
$$
If $\deg_p(g) \ge 1$, then $p^t \le \deg_p(g) \le \Delta$. For $2 \le m \le t+1$ we obtain as in \eqref{trivialtm}, 
\begin{equation*} 
|S(\chi,g,f,p^m)| \le p^{\frac {m}\Delta} p^{m(1-\frac 1\Delta)} \le p^{\frac {2t}\Delta} p^{m(1-\frac 1{\Delta})} \le \Delta^{\frac 2{\Delta}} p^{m(1-\frac 1\Delta)}\le 3^{\frac 23} p^{m(1-\frac 1\Delta)}.
\end{equation*} 
For $m \ge t+2$ we obtain as in \eqref{nontrivial},
\begin{align*} 
|S(\chi,g,f,p^m)| &\le 3\, p^{\frac tD}p^{m(1-\frac 1D)} 
\le 3\, p^{\frac t\Delta}p^{m(1-\frac 1\Delta)}
\le 3\, \Delta^{\frac 1\Delta} p^{m(1-\frac 1\Delta)} \le 3^{\frac 43}\, p^{m(1-\frac 1\Delta)}.
\end{align*}

\end{proof}

\section{Degenerate Sums, Proof of Theorem \ref{degentheorem}}
 \label{degensection}

 Let $f,g$ be rational functions over $\mathbb Q$, $p^m$ an odd prime power with $m \ge 2$ and $\chi$ a multiplicative character mod $p^m$ with $c=c_\chi$ and $p^{t_\chi}\|c$. Suppose that $\chi(g(x))e_{p^m}(f(x))$ is not a constant function on its domain. By translation if necessary we may assume that $g(0)$ is defined mod $p$ and that $p\nmid g(0)$.
Write
\begin{equation} \label{FGdef} 
f(X)=f(0)+p^{\ell_f}F(X), \quad \quad g(X)= g(0)(1+p^{\ell_g}G(X)),
\end{equation}
for some rational functions $F,G$ and non-negative integers $\ell_f, \ell_g$ with $p \nmid FG$.  For a degenerate sum we must have $\deg_p(f)=0$  and so $\ell_f>0$, as well as have $\ell_g+t_\chi>0$.  If $f(X)$ is not a constant polynomial mod $p^m$, then $\ell_f<m$, $F(X) \neq 0$, and $F(0)=0$. If $f(X)$ is a constant polynomial mod $p^m$, then (to avoid case studies) we take $\ell_f=m$, $F(X)=X$ so that in all cases $F(X) \neq 0$ and $X|F(X)$.
We do likewise for $g$, so $G(X) \neq 0$ and $X|G(X)$ in all cases.
In particular we may assume that $\deg_p(F) \ge 1$ and $\deg_p(G) \ge 1$.

\subsection{The case $\ell_g=0$} In this case, $\chi$ is imprimitive, that is $t_\chi>0$ and
 the value of $\chi(g(x))$ depends only on the value of $x$ mod $p^{m-t_\chi}$.  Thus, with $\ell:=\min\{t_\chi, \ell_f\}$, we have $1 \le \ell \le m$ and the value of $\chi(g(x))e_{p^m}(f(x))$ depends only on the value of $x$ mod $p^{m-\ell}$, implying
\begin{align}
S(\chi,g,f,p^m)&= 
p^{\ell} \sum_{x=1}^{p^{m-\ell}} \chi(g(x))e_{p^m}(f(0)+p^{\ell_f}F(x))\notag\\
&= p^{\ell}e_{p^m}(f(0)) \sum_{x=1}^{p^{m-\ell}} \chi(g(x))e_{p^{m-\ell}} (p^{\ell_f-\ell}F(x))\\
&=p^{\ell}e_{p^m}(f(0))\,  S(\chi,g,p^{\ell_f-\ell}F, p^{m-\ell})
, \label{Scasei}
\end{align}
where $\chi$ is now regarded as a  mod $p^{m-\ell}$ character  with $p^{t_\chi-\ell}\|c_\chi$.
If $\ell=t_\chi<\ell_f$, then $\chi$ is a primitive character mod $p^{m-\ell}$ and so the  sum in \eqref{Scasei} is non-degenerate, unless $m-\ell=1$ and $g(X)=bh(X)^r$ mod $p$ for some $h(X)$, where $r$ is the order of $\chi$. In the latter case, $\ell=m$ by definition, and the theorem holds trivially. 

If $\ell = \ell_f$ then $p \nmid p^{\ell_f-\ell}F(X)$ and so again the sum is non-degenerate.
To obtain the critical point function for the latter sum note that
$$
p^{\ell_f-\ell}F'(X) + p^{-\ell} c_{\chi} \frac {g'(X)}{g(X)}= p^{-\ell} \Big(f'(X) +c_\chi \frac {g'(X)}{g(X)}\Big).
$$
Thus $t(\chi, g,p^{\ell_f-\ell}F,p^{m-\ell}) = t-\ell$
 and the critical point function for the new sum is the same as for the original. 
 We have $\deg_p(p^{\ell_f-\ell}F)=\deg_p(F) \ge 1$ and obtain from Theorem \ref{maintheorem0} that
\begin{align} \label{llf}
|S(\chi,g,f,p^m)| \le p^{\ell} 3(\deg_p(F))^{\frac 1D} p^{(m-\ell)(1-\frac 1D)} =  3(\deg_p(F))^{\frac 1D} p^{\frac {\ell}D} p^{m(1-\frac 1D)}.
\end{align}
Similarly, if $\ell=t_\chi$, then $\chi$ is primitive mod $p^{m-\ell}$ and so by Theorem \ref{maintheorem0},
\begin{align} \label{ltchi}
|S(\chi,g,f,p^m)| &\le p^{\ell}\max\{3 \deg_p(g)^{\frac 1D},\, \deg_p(g)^{\frac 2D}\} p^{(m-\ell)(1-\frac 1D)}  \notag \\
& = \max\{3 \deg_p(g)^{\frac 1D},\, \deg_p(g)^{\frac 2D}\} p^{\frac {\ell}D} p^{m(1-\frac 1D)}.
\end{align}

\subsection{The case $\ell_g>0$} In this case, both $f$ and $g$ are constants mod $p$ and we have that $\log(g(X)/g(0))=\log(1+p^{\ell_g}G(X))$ admits a series expansion.
Set
\begin{align} \label{Hdef0}
H(X):= p^{\ell_f}F(X)+c_\chi\sum_{j=1}^J \tfrac {(-1)^{j-1}}{j} p^{j\ell_g}G(X)^{j} 
\end{align}
where $J$ is minimal such that for $j>J$, $j\ell_g-\text{ord}_p(j) \ge m$. Thus $H(X)$ is a rational function of degree at most $J\deg(G) +\deg(F)$. Moreover, by the choice of $J$,
\begin{align} \label{Hdef1} 
H(X) &\equiv p^{\ell_f}F(X) + c_\chi \log(1+p^{\ell_g}G(X)) \pmod {p^m} \notag\\ & \equiv f(X)-f(0)+c_\chi \log(g(X)/g(0)) \pmod {p^m}.
\end{align} 
Thus 
$$
H'(X) \equiv f'(X)+c_\chi g'(X)/g(X)\equiv p^{t}\mathcal C(X) \pmod {p^m}.
$$
If $p^m|H(X)$ (including $H(X)=0$)  put $\ell:=m$.  Otherwise
 put
$
\ell=\ell_H:=\text{ord}_p(H(X))$ and note that $\ell \ge 1$ since $\ell_f>0$ and $\ell_g+t_\chi>0$. 
For $\ell=m$ the estimate in the following proposition is trivial and so we assume henceforth that $\ell<m$.

By \eqref{plog}, for any integer $x$ we have 
\begin{align}\label{chiH}
\chi(g(x))e_{p^m}(f(x))&= \chi(g(0))\, e_{p^m}\big(\, c_\chi \log(1+p^{\ell_g}G(x))\big)\, e_{p^m}(f(0)) e_{p^m} \big(p^{\ell_f}F(x)\big) \notag\\
 &= \chi(g(0))\, e_{p^m}(f(0))\,  e_{p^m}(H(x)),
\end{align}
and so
\begin{align}
S(\chi,g,f,p^m)&= \chi(g(0))e_{p^m}(f(0))\sum_{x=1}^{p^m}e_{p^m}(H(x)) \notag\\
&= p^{\ell}\chi(g(0))e_{p^m}(f(0))\sum_{x=1}^{p^{m-\ell}}e_{p^{m-\ell}}(p^{-\ell}H(x)) \notag\\
&= p^{\ell}\chi(g(0))e_{p^m}(f(0))S(p^{-\ell}H, p^{m-\ell}). \label{Schiconvert1}
\end{align}

Thus
 $S(\chi,g,f,p^m)$ degenerates to the pure exponential sum $S(p^{-\ell}H, p^{m-\ell})$.
 The resulting sum is non-degenerate since $\deg_p(p^{-\ell}H) \ge 1$, and it has the same critical point function $p^{-t}(f'+c_\chi g'/g)$ as the original sum. The $t$-value for the new sum is $t(p^{-\ell}H,p^{m-\ell})=t-\ell$ since $H'(X)=p^t\mathcal C(X)$. 
 
Define $\Delta$ as in the introduction,
\begin{equation} \label{Deltadef2}
\Delta:=\begin{cases} \deg(f)+\deg(g), & \text{if $f,g$ are polynomials;}\\
2\deg(f)+2\deg(g),& \text{if $f$ or $g$ is non-polynomial.} \end{cases}
\end{equation}

\begin{proposition}
\label{degenprop} 

Let $f,g$ be rational functions over $\mathbb Q$, not both constants, $p^m$ a prime power, $\chi$ a  multiplicative character  mod $p^m$, and $\ell, t,D$ be  as defined  in \eqref{eldef}, \eqref{t}, \eqref{Ddef}. 

i) If $\ell_g=0$ then $S(\chi,g,f,p^m)$ degenerates to a  mixed character sum mod $p^{m-\ell}$  as given in \eqref{Scasei}, and
\begin{equation}
|S(\chi,g,f,p^m)|  \le \begin{cases} 3\, \deg_p(F)^{\frac 1D} p^{\frac {\ell}D} p^{m(1-\frac 1D)},& \text{if $\ell=\ell_f$;}\\
\max\{3\deg_p(g)^{\frac 1D},\, \deg_p(g)^{\frac 2D}\} p^{\frac \ell{D}} p^{m(1-\frac 1D)}; & \text{if $\ell=t_\chi$;}
\end{cases}
\end{equation}  
with $3$ replaced by $2^{\frac 53}$ for $p=2$. The same inequality holds with $\Delta$ in place of $D$.

ii) If $\ell_g>0$ then  $S(\chi,g,f,p^m)$ degenerates to a  pure exponential sum mod $p^{m-\ell}$  as in  \eqref{Schiconvert1}, and with $d_p=\deg_p(p^{-\ell}H)$, we have 
\begin{equation} \label{degenublg00}
|S(\chi,g,f,p^m)|  \le \max\{3d_p^{\frac 1D},\, d_p^{\frac 2{D}}\}\ p^{\frac {\ell}D} p^{m(1-\frac 1D)},
\end{equation} 
with $2^{\frac 53}$ in place of $3$ for $p=2$.
Again, the same holds with $\Delta$ in place of $D$. 
\end{proposition}

\noindent We prove the proposition in the next section. First we show how Theorem \ref{degentheorem} follows.

\begin{proof}[Proof of Theorem \ref{degentheorem}]  If $\ell_g=0$ and $\ell=\ell_f$,  then from the proposition
$$
|S(\chi,g,f,p^m)| \le 3 \deg(f)^{\frac 1\Delta} p^{\frac \ell \Delta} p^{m(1-\frac 1\Delta)} \le 3 \Delta^{\frac 1\Delta} p^{\frac \ell \Delta} p^{m(1-\frac 1\Delta)} \le 3^{\frac 43} p^{\frac \ell \Delta} p^{m(1-\frac 1\Delta)}. 
$$
The same holds if $\ell_g=0$ and $\ell=t_\chi$, since $\deg_p(g)^{\frac 1\Delta} \le 3\Delta^{\frac 1\Delta} \le 3^{\frac 43}$. 
If $\ell_g>0$ the statement of the proposition is the same as the theorem.
\end{proof}

\section{Proof of Proposition \ref{degenprop}}

\begin{lemma}\label{fprimedeg} Suppose that $q(X)=q_+(X)/q_-(X)$ is a rational function over a field $F$, with $q_+, q_-$ relatively prime polynomials over $F$. If $F$ has characteristic $p$,  assume further that $p>\max\{\deg(q_+),\deg(q_-)\}$. Let $Z$ denote the number of distinct zeros of $q_-$ in a splitting field. Then 
$$
\deg(q') =\max\{\deg(q_+)+Z-1,\ \deg(q_-)+Z\}\ge \deg(q)+Z-1.
$$  
\end{lemma}
\begin{proof} For simplicity of notation let $f=q_+$, $g=q_-$ so that $q'=\frac {f'g-g'f}{g^2}$ with $(f,g)=1$. Let $h=(f'g-g'f,g^2)$. Suppose that $\alpha$ is a zero of $g$ in a splitting field, with $(X-\alpha)^e\|g(X)$. Then $(X-\alpha)^{e-1}\|g'(X)$ since $p \nmid e$ in case $F$ has characteristic $p$. Since $q_+,q_-$ are relatively prime, $(X-\alpha) \nmid f(X)$ and so $(X-\alpha)^{e-1}\|(f'g-g'f)$. Thus if $g(X)=\prod_{\alpha \in \mathcal Z} (X-\alpha)^{e_\alpha}$ we have $h(X)= \prod_{\alpha \in \mathcal Z} (X-\alpha)^{e_\alpha-1} $, a polynomial of degree $\deg(g)-Z$,  and the numerator of $q'$  in reduced form has degree $\deg(f'g-g'f)-(\deg(g)-Z)$ while the denominator has degree $2\deg(g)-(\deg(g)-Z)= \deg(g)+Z$. 

Suppose that $\deg(f) \neq \deg(g)$ so that $\deg(f'g-g'f)=\deg(f)+\deg(g)-1$. Then in reduced form, the numerator of $q'$ has degree $\deg(f)+\deg(g)-1-(\deg(g)-Z)=\deg(f)+Z-1$ yielding the desired equality.  If $\deg(f)=\deg(g)$ then the numerator can have smaller degree, but in this case the denominator has maximum degree, $\deg(g)+Z$, and again the result follows.
\end{proof}

\begin{lemma} \label{dpHD} i) Let $H$ be as defined in \eqref{Hdef}. Then $p^{t-\ell}\le \deg_p(p^{-\ell} H)$.

ii) If $p>\deg_p(p^{-\ell}H)$, then $\deg_p(p^{-\ell}H) \le D$, with $D$ as in \eqref{Ddef}.
\end{lemma}

\begin{proof} $i)$ Since $H'(X)= p^t\mathcal C(X)$ we have $ p^{t-\ell}|p^{-\ell}H'(X)$ and so by Corollary \ref{tboundcor}, $p^{t-\ell}\le \deg_p(p^{-\ell}H)$. 

$ii$) Since $p> \deg_p(p^{-\ell}H)$ it follows from part $(i)$ that $\ell=t$. Write $p^{-\ell}H(X)=\frac {p^{-\ell}H_+(X)}{H_-(X)}$ say in reduced form over $\mathbb Z$ and $p^{-\ell}H(X) \equiv q(X)= \frac {q_+(X)}{q_-(X)}$ mod $p$ in reduced form over $\mathbb F_p$.  Suppose first that $\deg_p(q_+)\ge \deg_p(q_-)+1$ so that by Lemma \ref{fprimedeg}, $\deg_p(p^{-\ell}H') = \deg_p (q_+)+Z-1$ where $Z$ is the number of distinct zeros of $q_-$. But $p^{-\ell}H'(X)\equiv \mathcal C(X)$ mod $p$ and so 
$$
\deg_p(q)=\deg_p(q_+)=  \deg_p(\mathcal C_+) -Z+1 \le \deg(\mathcal C)-Z+1 \le D-Z,
$$
 by Lemma \ref{Cxdegub}.

If $\deg_p(q_+)\ge \deg_p(q_-)+1$ then by Lemma \ref{fprimedeg}, $\deg_p(p^{-\ell}H') = \deg_p (q_-)+Z$, and so 
$$
\deg_p(q)=\deg_p(q_-)=  \deg_p(\mathcal C_-) -Z \le \deg(\mathcal C_-)-Z \le D-Z.
$$
\end{proof} 

\begin{proof}[Proof of Proposition \ref{degenprop}] $i$) We may assume that $\ell<m$.  The first bound in part $(i)$  was established in \eqref{llf} and \eqref{ltchi} for odd $p$. 
The second estimate follows in the same manner using the bound from Corollary \ref{maincor}, $|S(\chi,g,f,p^m)| \le 3^{\frac 43} p^{m(1-\frac 1\Delta)}$, instead of Theorem \ref{maintheorem0}. Thus for $\ell=\ell_f$ we get
$$
|S(\chi,g,f,p^m)| \le p^{\ell} 3^{\frac 43} p^{(m-\ell)(1-\frac 1\Delta)}=3^{\frac 43} p^{\frac \ell \Delta} p^{m(1-\frac 1\Delta)},
$$
and similarly for $\ell=t_\chi$.

$ii$) Suppose now that $\ell_g>0$ so that we have the decomposition in \eqref{Schiconvert1}.  Let $d_p:=\deg_p(p^{-\ell} H)$. 
By Lemma \ref{dpHD},
 $p^{t-\ell} \le d_p$, 
that is $p^t  \le p^{\ell} d_p$. For $1 \le m \le t$ we then have trivially
$$
|S(\chi,f,g,p^m)| \le p^m \le p^{\frac {t}D}p^{m(1-\frac 1D)} \le p^{\frac \ell{D}}  d_p^{\frac 1D} p^{m(1-\frac 1D)}.
$$
Similarly if $m=t+1$ and $p \le d_p$ then
$$
|S(\chi,f,g,p^m)| \le  p^{\frac {t+1}D}p^{m(1-\frac 1D)} \le p^{\frac \ell{D}} p^{\frac 1D} d_p^{\frac 1D} p^{m(1-\frac 1D)} \le  p^{\frac \ell{D}}  d_p^{\frac 2D} p^{m(1-\frac 1D)}.
$$ 

Suppose now that $m=t+1$ and $p>d_p$. Then, by  Lemma \ref{dpHD}, $\ell=t$ and $d_p:=\deg_p(p^{-\ell}H) \le D$. 
Now, trivially
$$
|S(\chi,f,g,p^m)| \le p^m = p^{\ell+1}  \le 3 d_p^{\frac 1D}p^{\frac \ell D} p^{m(1-\frac 1D)},
$$
provided that $p<3^D d_p$.  Otherwise, $p>3^D d_p$ and we obtain from \eqref{Schiconvert1} and
 the Weil estimate for  $S(p^{-\ell} H,p)$,  that
$$
|S(\chi,f,g,p^m)| \le p^{\ell}(2d_p)\sqrt{p} \le 3 p^{\frac {\ell}D}d_p^{\frac 1D} p^{m(1-\frac 1D)}, 
$$
provided that $d_p \le \frac 14 3^D$, which is the case since $d_p \le D \le \frac 14 3^D$ for $D \ge 2$.

Finally,
if $m \ge t+2$ then by Corollary \ref{m2} and  the fact that $p^t \le p^{\ell}d_p$, we get
\begin{align*}
|S(\chi,f,g,p^m)|
& \le  3\, p^{\frac {t}D} p^{m(1-\frac 1D)}\le 3\, p^{\frac {\ell}{D}}d_p^{\frac 1D} p^{m(1-\frac 1D)}. 
\end{align*}

\end{proof}

 \section{Estimating $\deg_p(p^{-\ell} H)$}

Consider several cases. If $\ell_f \neq \ell_g+t_\chi$, then the degree is easy to determine.
\begin{lemma}   Suppose that $p$ is odd, $\ell_f>0$, $\ell_g>0$ and $\ell_f \neq \ell_g+t_\chi$.  Then,
 $$
 \deg_p(p^{-\ell}H) = \begin{cases} \deg_p(F), & \text{if $\ell_f<\ell_g+t_\chi$;}\\
 \deg_p(G), & \text{if $\ell_f>\ell_g+t_\chi$.}
 \end{cases}
 $$
\end{lemma}

\begin{proof} By \eqref{Hdef1},  we have $\ell \ge \min\{\ell_f, \ell_g+t_\chi\}$ with equality if the two values are different. To be precise,  
if $\ell_f<\ell_g+t_\chi$ then $\ell=\ell_f$,  $p^{-\ell}H(X) \equiv F(X)$ mod $p$, and $D^*=\deg_p(F)$. If $\ell_f>\ell_g+t_\chi$, then $\ell= \ell_g+t_\chi$, $p^{-\ell}H(X) \equiv c_\chi p^{-t_\chi}G(X)$  and $D^*=\deg_p(G)$.
\end{proof}

Suppose now that $\ell_f=\ell_g+t_\chi$.
 In particular, $\ell \ge \ell_f$ and $\ell-t_\chi\ge \ell_g \ge 1$.
Let $L$ be the minimal positive integer such that
\begin{equation} \label{Ldef}
\text{ord}_p(c_\chi p^{j\ell_g}/j)=j\ell_g+t_\chi-\text{ord}_p(j) > \ell, \quad \quad \text{for $j>L$},
\end{equation}
so that
\begin{align}
p^{-\ell}H(X) &\equiv p^{-\ell}\Big( p^{\ell_f}F(X)+p^{-t_\chi}c_\chi \sum_{j=1}^L \tfrac {(-1)^{j-1}}{j} p^{j\ell_g} G(X)^{j} \Big), \pmod {p}. \label{plHX}
\end{align}
By the definition of $J$ \eqref{Hdef0}, we have $L \le J$. 
From \eqref{Ldef} it follows that
$
L\le \tfrac 1{\ell_g}\big(\ell-t_\chi+\text{ord}_p(L)\big),
$
and  $L+1>\frac 1{\ell_g}\big(\ell-t_\chi+\text{ord}_p(L+1)\big)$.
Since $\frac {\text{ord}_p(L)}{L} \le \frac 1p$, we obtain
\begin{equation} \label{Lsqueeze}
\frac {\ell-t_\chi}{\ell_g}-1< L \le \frac {\ell-t_\chi}{\ell_g}\cdot \frac p{p-1}.
\end{equation}

\begin{lemma} \label{L1} Suppose that $p$ is odd, $\ell_f>0$, $\ell_g>0$ and $\ell_f=\ell_g+t_\chi$.  

i) If $\ell=\ell_f$ then $L=1$.

ii) If $L=1$ then for any $\ell$, 
$$
\deg_p(p^{-\ell}H)\le  \deg_p(F)+\deg_p(G)
$$
with improvement to $\max\{\deg_p(F), \deg_p(G)\}$ for polynomial $F,G$.
\end{lemma}

\begin{proof}
We have $L=1$ if and only if $t_\chi+2\ell_g>\ell$ that is, $\ell_f+\ell_g > \ell$. In particular, if $\ell=\ell_f$ then  $L=1$ (since $\ell_g>0$). If $L=1$ we have 
\begin{equation} \label{ellellf}
p^{-\ell}H(X) \equiv F(X) +p^{-t_\chi}c_\chi\, G(X) \pmod p,
\end{equation}
and so $\deg_p(p^{-\ell}H)\le \deg_p(F)+\deg_p(G)$.
\end{proof}
 From \eqref{plHX}, we immediately deduce the next lemma.

 \begin{lemma}\label{DstarUB}  Suppose that $p$ is odd, $\ell_f>0$, $\ell_g>0$ and  $\ell_f=\ell_g+t_\chi$. Then,
 $$
 \deg_p(p^{-\ell}H) \le 
 \deg(F)+L\deg(G),
       $$
 with improvement to $\max\{\deg(F),\, L\deg(G)\}$ if $F$ and $G$ are both polynomials. 
 \end{lemma}
\noindent 
 From this lemma and the upper bound for $L$ in \eqref{Lsqueeze}, we have
\begin{equation} \label{dpHub}
\deg_p(p^{-\ell}H) \le \deg(F)+ \tfrac {\ell-t_\chi}{\ell_g} \tfrac p{p-1} \deg(G).
\end{equation}
The drawback of this bound is that it involves the parameter $\ell$ on the right-hand side.  In order to remove this parameter we proceed as follows.

 Suppose that $\ell_f=\ell_g+t_\chi$. In particular $\ell \ge \ell_f$. 
  By the definition of $L$,
$$
p^{\ell}|\big(p^{\ell_f} F(X)+ c_\chi\sum_{j=1}^L 
\tfrac {(-1)^{j-1}}{j} p^{j\ell_g}G(X)^{j}\big),
$$
and 
 \begin{align}
F(X) &\equiv  -p^{-t_\chi}c_\chi \sum_{j=1}^L 
\tfrac {(-1)^{j-1}}j p^{(j-1)\ell_g} G(X)^{j} , \pmod {p^{\ell-\ell_f}}. \label{FG}
\end{align} 
Thus it is reasonable to conjecture that in general the degree of $F$  is of order $L \deg(G)$.  
Whenever we have such a relationship, we obtain a good bound on $\deg_p(p^{-\ell}H)$.  

\begin{lemma} If $\deg(F) \ge (L-2)\deg(G)$, then $\deg(p^{-\ell}H) \le 2\deg(F)+2\deg(G)$, with improvement to $\deg(F)+2\deg(G)$ in case $F$ and $G$ are polynomials.
\end{lemma}
 
\begin{proof}
 If $(L-2)\deg(G) \le \deg(F)$ then by Lemma \ref{DstarUB}
\begin{align*}
\deg_p(p^{-\ell}H) &\le \deg(F)+L\deg(G) = \deg(F)+(L-2)\deg(G) +2\deg(G)\\
&  \le 2\deg(F) +2 \deg(G),
\end{align*}
with improvement to $\deg(F)+2\deg(G)$ in case $F,G$ are polynomials.
\end{proof}

\section{Laurent Polynomials, proof of Theorem \ref{Laurenttheorem}}
Suppose that $F(X)$ and $G(X)$ are  Laurent polynomials, say
\begin{equation} \label{GXlaurent}
G(X)=\sum_{n=d_1}^{d_2} a_nX^n, \text{with $a_n \in \mathbb Q$, $\text{ord}_p(a_n) \ge 0$, $d_1 \le n  \le d_2$, \  $a_{d_1} \neq 0$,\  $a_{d_2} \neq 0$}.
\end{equation}

\begin{lemma} Suppose that $\ell_g+t_c=\ell_f<\ell$ and that $F(X)$ and $G(X)$ are Laurent polynomials  over $\mathbb Z$ with $G(X)$ as given in \eqref{GXlaurent}.  If $|d_2|\ge |d_1|$ assume $p \nmid a_{d_2}$. If $|d_2|<|d_1|$ assume $p \nmid a_{d_1}$. 
Then
$$
\deg(F) \ge (L-2)\deg(G),
$$
\end{lemma}
\begin{proof} We may assume that $L \ge 3$ and  that $c_\chi p^{-t_\chi}=1$ so that \eqref{FG} gives
\begin{equation} \label{Fexpansion}
F(X)= G(X)-\tfrac 12 p^{\ell_g}G(X)^2+\cdots +\tfrac {(-1)^{L-1}}{L} p^{(L-1)\ell_g}G(X)^L+p^{\ell-\ell_f}J(X),
\end{equation}
for some Laurent polynomial $J(X)$ over $\mathbb Z_p$.  We have 
$$
\ell-\ell_f \ge L\ell_g+t_c-\text{ord}_p(L)-\ell_f =  (L-1)\ell_g-\text{ord}_p(L).
$$

Suppose first that $|d_2| \ge |d_1|$, so that $d_2>0$, $p \nmid a_{d_2}$ and $\deg(G)=d_2$.
Consider the coefficient of $X^{d_2L}$ in the expansion \eqref{Fexpansion}. Among the first $L$ terms involving powers of $G(X)$, this term appears once (in the $G(X)^L$ term)  with coefficient $\frac {(-1)^{L-1}}L p^{(L-1)\ell_g}$ of $p$-order $(L-1)\ell_g-\text{ord}_p(L)\le \ell-\ell_f$ (since $p \nmid a_{d_2}$).  It can also appear in the final term (with $J(X)$)  with coefficient having $p$-order at least $\ell-\ell_f$.  If the $p$-orders for these two terms are different, that is $(L-1)\ell_g-\text{ord}_p(L)< \ell-\ell_f$, or the $p$-orders are the same and the two terms do not cancel, then  $F(X)$ has a nonzero term of degree $Ld_2$.

Suppose now that the $p$-orders of the two terms are the same  and that the $X^{d_2L}$ term coming from $G(X)^L$ cancels with the $X^{d_2L}$ term coming from $J(X)$. In this case, consider the coefficient of $X^{d_2(L-1)}$. This term appears in the $G(X)^{L-1}$ term with coefficient of $p$-order $(L-2)\ell_g-\text{ord}_p(L-1)$ which will be smaller than $(L-1)\ell_g-\text{ord}_p(L)$
unless $p|L$. If $p \nmid L$, we conclude that $F(X)$ has a nonzero term $X^{d_2(L-1)}$ since appearances in the $G(X)^L$ term or in the $J(X)$ term have larger $p$-orders.

Suppose finally, that $p|L$ and that $(L-1)\ell_g-\text{ord}_p(L)=(L-2)\ell_g-\text{ord}_p(L-1)=(L-2)\ell_g$, that is, $\ell_g=\text{ord}_p(L)$. In this case we turn our attention to the $X^{d_2(L-2)}$ term, which will have a coefficient with $p$-order strictly less than the $p$-order of the terms to the right and thus it does not vanish. Thus in all three cases we see that $F(X)$ has a term of positive degree at least $(L-2)d_2$.  
Since $d_2=\deg(G)$, we are done.

Suppose now that $|d_2|<|d_1|$, so that $d_1<0$, $p \nmid a_{d_1}$ and $\deg(G)=|d_1|$.
 Applying the same argument as above to the minimal degree term $a_{d_1}X^{d_1}$ we see that $F(X)$ has a term of negative degree  $ \le (L-2)d_1$. Thus $\deg(F) \ge (L-2)|d_1|=(L-2)\deg(G)$ as desired. 
\end{proof}

\begin{proof}[Proof of Theorem \ref{Laurenttheorem}]
For any Laurent polynomials $F,G$ satisfying the hypotheses of the theorem, 
we obtain from the preceding two lemmas   that
\begin{equation}
d_p:=\deg_p(p^{-\ell}H)   \le   2\deg(F)+2\deg(G),
\end{equation}
with improvement to $2\deg(F)+\deg(G)$ in case $F$ and $G$ are polynomials.
Thus $d_p \le 2\Delta$ for polynomial $F,G$ and $d_p \le \Delta$ for non-polynomial $F,G$.
Then by \eqref{degenublg00} (with exponent $\Delta$ in place of $D$)  we get for non-polynomial $f,g$, 
$$
|S(\chi,g,f,p^m)| \le \max \{3(\Delta)^{\frac 1 \Delta}, \,  (\Delta)^{\frac 2\Delta}\}\,  p^{\frac \ell \Delta} p^{m(1-\frac 1\Delta)} \le 3^{4/3} p^{\frac \ell \Delta} p^{m(1-\frac 1\Delta)}.
$$
For polynomial $f,g$, we obtain the same  with constant $6$. 
\end{proof}

\section{The prime $p=2$}\label{p2section}

Suppose that $m \ge 3$.
Let $\chi$ be a multiplicative character mod $ {2^m}$ defined by the
relations  
\begin{equation} \label{c2}
\chi(5)=e_{2^{m-2}}(c), \quad \chi(-1)=(-1)^\kappa,
\end{equation}
for some integer $c$ with $1 \le c \le 2^{m-2}$ and $\kappa =0$ or 1.
Let $R$ be the 2-adic integer
\begin{equation} \label{R2}
R:= \frac 14 \log (5) = \sum_{j=1}^\infty \frac {(-1)^{j-1} 4^{j-1}}{j}
\equiv
-1 \pmod {16},
\end{equation}
and set $c_\chi:=\overline Rc$, $t_\chi= \text{ord}_2(c_\chi)$.
 Then we have the following analogue of \eqref{plog}, valid for any 2-adic integer $y$:
\begin{align} \label{chiuntwist2}
\chi (1+4y) &= e_{2^{m}}\left(c_\chi \log (1+4y)\right)\\
\chi(-1+4y) &= (-1)^\kappa\, e_{2^{m}}\left(c_\chi\log (1-4y)\right)\notag.
\end{align}

 Let $f,g$ be rational functions as in \eqref{fg}, not both constants,  with $\text{ord}_2(f) \ge 0$, $\text{ord}_2(g) =0$, and
$S(\chi,g,f,2^m)$ be the exponential sum in \eqref{sum}. As before, let
 $t=t(\chi,g,f,2^m)$ be defined by
$2^t\| (f'+c_\chi g'/g)$, and
\begin{equation} \label{cpc2}
\mathcal C(X):= 2^{-t}\Big(f'(X) +c_\chi\frac {g'(X)}{g(X)}\Big).
\end{equation}
A value $\alpha  \in \{0,1\}$ is called a critical point if $\mathcal C(\alpha) \equiv 0$ mod $2$.

\begin{proposition} \label{propconvert2} Suppose that $f,g$ are rational functions over $\mathbb Q$, not both constants with $\text{ord}_2(f) \ge 0$, $\text{ord}_2(g) = 0$, and that $m \ge t+3$. Let $\alpha \in \{0,1\}$. If $\alpha$ is not a critical point for the sum $S(\chi,g,f,2^m)$, then $S_\alpha=0$.   If $\alpha$ is a critical point, then $S_\alpha=S_{\alpha,1}+S_{\alpha,2}$, with
\begin{equation} 
S_{\alpha,1}= \begin{cases} 2^{\sigma -2}\chi(g(\alpha)) e_{2^m}(f(\alpha))
S(G_{\alpha},2^{m-\sigma}), & \text{if $m>\sigma$;}\label{sg2}\\
 2^{m-3}\chi(g(\alpha)) e_{p^m}(f(\alpha)), & \text{if $m \le \sigma$;}\end{cases}
\end{equation}
where $\sigma$ and $G_{\alpha}$ are  as defined in  \eqref{sigmaalpha1} and \eqref{galpha1},
 and $S_{\alpha,2}$  is the same with $\alpha$ replaced by $\alpha+2$, and $G_{\alpha}$ by $G_{\alpha+2}$ as given in \eqref{galpha2}.
\end{proposition}

\begin{proof} Suppose that $m \ge t+3$. For any positive integer $u$, with $2 \nmid g(u)$,  define
\begin{equation} \label{F2}
F_u (Y) := c_\chi \log \left( \frac {g(u
+2Y)}{g(u)}\right)
+ f(u +2Y) -f(u),
\end{equation} 
so that
\begin{equation} \label{F2prime}
F_u'(Y)= 2^{t+1} \mathcal C(u+2Y),
\end{equation}
where $\mathcal C$ is the critical point function \eqref{cpc2}.
Write $x=u+2^{m-t-1}v$ with $u$ running from 1 to $2^{m-t-1}$ and
$v$ running from 1 to $2^{t+1}$. Then, noting that $m_1:=m-t-1 \ge 2$, using \eqref{chiuntwist2} and following what we did for the case of odd $p$, we have
\begin{align*}
&S(\chi,g,f,2^m)=\sum_{u=1}^{2^{m_1}}\  \sum_{v=1}^{2^{t+1}} \chi(g(u+2^{m-t-1}v))e_{2^m}(f(u+2^{m-t-1}v))\\
&= \sum_{u=1}^{2^{m_1}} \chi(g(u))e_{2^m}(f(u)) \sum_{v=1}^{2^{t+1}} e_{2^m}\Big(c_\chi \log\big(\frac {g(u+2^{m-t-1}v)}{g(u)}\big) +f(u+2^{m-t-1}v)-f(u)\Big). 
\end{align*}
 Thus by the definition of $F_u$ we have
 \begin{equation} \label{Schiconvert2}
S(\chi,g,f,2^m)=\sum_{u=1}^{2^{m-t-1}} \chi(g(u))e_{2^m}(f(u)) \sum_{v=1}^{2^{t+1}} e_{2^m}\Big(F_u(2^{m-t-2}v)\Big) .
\end{equation}
Now, by \eqref{F2prime}, $2^{t+1}|F_u'(Y)$, and $2^{t+2}|F_u^{(k)}(Y)$ for $k \ge 2$, and so for $m \ge t+3$, it follows by a Taylor series expansion, that
$$
F_u(2^{m-t-2}v) \equiv F_u(0)+F_u'(0)2^{m-t-2}v  \pmod {2^m}.
$$
Since $F_u(0)=0$ and $F_u'(0) =2^{t+1}\mathcal C(u)$, we get
$$
F_u(2^{m-t-2}v) \equiv 2^{m-1} \mathcal C(u)\, v \pmod {2^m}.
$$ 
Thus, the sum over $v$ in \eqref{Schiconvert2} vanishes unless $2|\mathcal C(u)$, that is, $u\equiv \alpha$ mod $2$, for some critical point $\alpha$. In the latter case, the sum over $v$ is equal to $2^{t+1}$.  Thus
\begin{equation} \label{Schi2}
S(\chi,g,f,2^m)=\sum_{\alpha \in \mathcal A} S_\alpha(\chi,g,f,2^m) = 2^{t+1} \sum_{\alpha \in \mathcal A}\  \sum_{u\equiv \alpha \ \text{mod}\ 2}^{2^{m-t-1}} \chi(g(u))e_{2^m}(f(u)).
\end{equation}

Suppose now that $\alpha\in \{0,1\}$ is a fixed critical point. To compute $S_\alpha(\chi,g,f,2^m)$,
 first write $u=\alpha+4y$ with $y$ running from 1 to $2^{m-t-3}$ to get
\begin{align}
S_{\alpha,1}(\chi,g,f,p^m)&:= 2^{t+1} \sum_{y=1}^{2^{m-t-3}}\chi (g(\alpha+4y)) e_{2^m}(f(\alpha +4y))\notag \\
&=  2^{t+1}\chi(g(\alpha)) e_{2^m}(f(\alpha))
\sum_{y=1}^{2^{m-t-3}} e_{2^m} (F_\alpha(2y)). \label{SalphaF1} 
\end{align}   
Expand
$F_\alpha(Y)$ into a formal power series 
\begin{equation} \label {aj1}
F_\alpha (Y)= \sum_{j=1}^\infty a_j Y^j,
\end{equation}
with $p$-adic integer coefficients $a_j$, and
define
\begin{equation}
 \label{sigmaalpha1}
\sigma:= \text{ord}_p(F_\alpha(2Y))= \min_{j\ge 1} \{\text{ord}_p(2^ja_j)\},
\end{equation}
\begin{equation} \label{galpha1}
G_{\alpha}(Y):=
2^{-\sigma}F_\alpha (2Y),
\end{equation}
so that
$$
S_{\alpha,1}(\chi,g,f,2^m)= 2^{\sigma-2}\chi(g(\alpha))e_{2^m}(f(\alpha)) S(G_\alpha, 2^{m-\sigma}).$$

Next we write $u = \alpha+2+4y$ with $y $ running from 1 to $2^{m-t-3}$, and define
\begin{align*}
S_{\alpha,2}(\chi,g,f,p^m)&:= 2^{t+1} \sum_{y=1}^{2^{m-t-3}}\chi (g(\alpha+2+4y)) e_{2^m}(f(\alpha+2 +4y)),
\end{align*}
so that $S_\alpha = S_{\alpha,1}+S_{\alpha,2}$. Set 
\begin{equation}
 \label{sigmaalpha2}
\sigma'=:= \text{ord}_p(F_{\alpha+2}(2Y)),
\end{equation}
\begin{equation} \label{galpha2}
G_{\alpha+2}(Y):=
2^{-\sigma'}F_{\alpha+2} (2Y).
\end{equation}
 We obtain as before
$$ 
 S_{\alpha+2}(\chi,g,f,2^m)= 2^{\sigma-2}\chi(g(\alpha+2))e_{2^m}(f(\alpha+2)) S(G_{\alpha+2}, 2^{m-\sigma'}),
 $$
 completing the proof of the proposition.

\end{proof}

\subsection{Relations between parameters for $p=2$} 
Let $\alpha \in \{0,1\}$ be a given  critical point, $F_\alpha(Y)=\sum_{j=1}^\infty a_jY^j$, $G_\alpha(Y)=2^{-\sigma}F_\alpha(2Y)$ and write
$$
\mathcal C(Y):=\sum_{j=0}^{\infty} c_j(Y-\alpha)^j,
$$
for some $2$-adic integer coefficients $c_j$.
Define $\tau$ by $2^\tau\|G_\alpha'(Y)$. Then since $F_\alpha'(Y)=2^{t+1}\mathcal C(\alpha+2Y)$, we obtain
\begin{align}
F_\alpha(Y) &= \sum_{j=1}^\infty a_jY^j =2^{t} \sum_{j=1}^\infty
c_{j-1}2^j\frac {Y^j}j, \label{Falpha2}  \\
G_\alpha (Y) &= 2^{-\sigma} \sum_{j=1}^\infty 2^ja_jY^j =2^{t-\sigma}\sum_{j=1
}^{\infty}\frac {c_{j-1}}j 4^jY^j,\label{galpha22}  \\
H_\alpha (Y) &:=2^{-\tau}G_\alpha'(Y)= 2^{-\tau-\sigma} \sum_{j=1}^\infty 2^ja_jjY^{j-1} =
2^{t-\tau-\sigma} \sum_{j=1}^\infty c_{j-1}4^jY^{j-1}.\label{halpha2}
\end{align}

Suppose that $\alpha$ is a critical point of multiplicity $\nu \ge 1$.
Arguing as in the case of odd $p$, we have the following relations for
$p=2$:
\begin{align}
\label{sigmalow2}
&\sigma \ge t+3,\\
\label{sigmaup2}
&\sigma \le 2\nu+2+t-\tau, \\
\label{dpg22}
&d_p(G_\alpha) \le \frac 12 \big(\sigma -t +\text{ord}_2(d_p(G_\alpha))\big),\\
\label{dph2}
&d_p(H_\alpha) \le \frac 12(\sigma +\tau-t)-1 \le \nu,\\
\label{tauup2}
&\tau \le \text{ord}_2(d_p(G_\alpha)).
\end{align}
\noindent We note that the same relations hold when $\alpha$, $\sigma$, $F_\alpha$, and $G_\alpha$ are replaced by $\alpha+2$, $\sigma'$, $F_{\alpha+2}$ and $G_{\alpha+2}$.

\section{Main Theorem for $p=2$}
\begin{theorem}\label{maintheoremp2}
Let $f,g$ be rational functions over $\mathbb Q$, not both constants, and $\chi$ be a multiplicative character
mod ${2^m}$ such that
$m\ge t+3$. Let $\alpha \in \{0,1\}$.

\noindent
i) If $\alpha$ is not a critical point, then $S_\alpha=0$.

\noindent
ii) If $\alpha$ is a critical point of multiplicity 1,  then $|S_\alpha|\le 2^{\frac {m+t+1}2}$.  If in addition, $m \ge t+5$, 

\hskip .1in then $|S_\alpha|=2^{\frac {m+t}2}$.

\noindent
iii) If $\alpha$ is a critical point of multiplicity $\nu \ge 1$, then
\begin{equation} \label{ub22}
|S_\alpha(\chi,g,f,2^m)|\le 2^{\frac 53} \cdot 2^{\frac t{\nu+1}}2^{m(1-\frac
1{\nu+1})}.
\end{equation}

\end{theorem}

\begin{corollary} \label{maincorp21} Under the hypotheses of Theorem \ref{maintheoremp2}, if $m \ge t+3$ and $d=\deg_2(\mathcal C^+)$ where $\mathcal C^+$ is the numerator of $\mathcal C(X)$, then
$$
|S(\chi,g,f,2^m)| \le 2^{\frac 53}\cdot 2^{\frac t{d+1}} 2^{m(1-\frac 1{d+1})}.
$$
\end{corollary}  

\begin{corollary}\label{maincor2p2} Suppose that $f=f_+/f_-$, $g=g_+/g_-$ are rational functions over $\mathbb Q$ with $2 \nmid f_-g_+g_-$, $m$ is a positive integer and
$\chi$ is a multiplicative character mod ${2^m}$. If $m=1$ suppose that either $\deg_2(f) \ge 1$ or that $g$ is not of the form $g(X) \equiv bh(X)^r$ mod $2$ for some integer $b$ and rational function $h(X)$, where $r$ is the order of $\chi$.  If $m \ge 2$  suppose that either $\deg_2(f) \ge 1$, or that $\deg_2(g) \ge 1$ and $\chi$ is primitive. Then we
have
\begin{equation} \label{uniformp2}
|S(\chi,g,f,2^m)|\le 2^{\frac 53}\, D^{\frac 1D} 2^{m\big(1-\frac 1{D}\big)}.
\end{equation}
\end{corollary}

\section{Proof of Theorem  \ref{maintheoremp2}}     
 Part $(i)$ was established in Proposition \ref{propconvert2}.  Part $(ii)$ was established in \cite[Theorem 6.1]{cochrane2002}. We are left with establishing part $(iii)$ for  a critical point $\alpha$ of multiplicity $\nu \ge 2$. Suppose that $m \ge t+3$.
 
Case $i$. Suppose that $\sigma \ge m-2-\tau$.
Then we have trivially,
\begin{align*}
|S_\alpha|&\le 2^{m-1}=2^{-1}2^{\frac {m}{\nu+1}}2^{m(1-\frac 1{\nu+1})}\le 2^{-1}2^{\frac {\sigma+2+\tau}{\nu+1}}2^{m(1-\frac 1{\nu+1})}\\
&\le 2 \cdot 2^{\frac {2}{\nu+1}}2^{\frac t{\nu+1}} 2^{m(1-\frac 1{\nu+1})} ,
\end{align*}
the latter inequality following from \eqref{sigmaup2}. Since $2^{\frac 2{\nu+1}} \le 2^{\frac 23}$, the result follows.

Case $ii$. Suppose that $\sigma \le m-3-\tau$, that is, $m-\sigma \ge \tau+3$. First, by Proposition \ref{propconvert2}, we have $|S_\alpha| \le |S_{\alpha,1}|+|S_{\alpha, 2}|$, with
$$
|S_{\alpha,1}| = 2^{\sigma-2}|S(G_\alpha, 2^{m-\sigma})|,
$$
and the same for $|S_{\alpha,2}|$ with $\alpha$ replace by $\alpha +2$.
Since $m-\sigma \ge \tau+3$, we  can apply Theorem \ref{mainthm1}  to
$S(G_\alpha,2^{m-\sigma})$ and
obtain with $d_1:=\deg_2(H_\alpha)=\deg_2(2^{-\tau}G_\alpha')$,   
$$   
|S_{\alpha,1}|  \le
2^{2/3}\, 2^{\sigma
-2} 2^{\frac {\tau}{d_1+1}}2^{(m-\sigma)(1-\frac
1{d_1+1})}.
$$
The same bound holds for $|S_{\alpha,2}|$.
Now  by \eqref{dph2},
$d_1=\deg_p(H_\alpha)\le \nu$ and thus since $m-\sigma -\tau >0$ we obtain, using \eqref{sigmaup2} again, 
\begin{align*}
|S_{\alpha}|&\le 2 \cdot 2^{2/3}\, 2^{\sigma-2}2^{\frac
{\tau}{\nu+1}}2^{(m-\sigma)(1-\frac
1{\nu+1})}=2^{-1/3}\, 
 2^{\frac {\tau+\sigma}{\nu+1}} 2^{m(1-\frac 
1{\nu+1})}\\
&  \le 2^{-1/3} 2^{\frac {\tau+2\nu+2+t-\tau}{\nu+1}} 2^{m(1-\frac 1{\nu+1})} = 2^{5/3}2^{\frac {t}{\nu+1}} 2^{m(1-\frac
1{\nu+1})}.
\end{align*}

\section{Proof of Corollary \ref{maincorp21}}
Put $d=\deg_p(\mathcal C_+)$, $S=S(\chi,g,f,2^m)$. The bound in \eqref{Cxdegub} holds for $p=2$ and so we again have $d \le D-1$. Assume that $m \ge t+3$. 
  For $1 \le j \le d-1$, let $n_j$ denote the number of critical points of multiplicity $j$ and put $x_j:=jn_j$, so that $\sum_{j=1}^{d}x_j\le d$. Put $\delta := 2^{t-m}$.  Following the proof of Theorem \ref{maintheorem} but using Theorem \ref{maintheoremp2} $(iii)$ instead,  we see that
  $$
  |S(\chi,g,f,2^m)|\le 2^{\frac 53} \cdot 2^{\frac t{d+1}} 2^{m(1-\frac 1{d+1})},
  $$
for $2^{t-m} \le (1+\frac 1{d-1})^{-d(d+1)}$, that is,  $2^{\frac {m-t}{d+1}} \ge (1+\frac 1{d-1})^d$.  Suppose now that $2^{\frac {m-t}{d+1}}<(1+\frac 1{d-1})^d$. We have trivially
\begin{equation} \label{2trivial}
|S| \le 2^m \le  2^{\frac 53} \cdot 2^{\frac t{d+1}} 2^{m(1-\frac 1{d+1})}
\end{equation}
provided that $2^{\frac {m-t}{d+1}} \le 2^{\frac 53}$. Thus it suffices to have $(1+\frac 1{d-1})^d \le 2^{\frac 53}$ which is the case for $d \ge 4$.

We are left with the cases where $d$ equals $1,2$ or 3. Since $p=2$, there are at most two critical points and we may assume that the sum of their multiplicities is $d$.
 If there is a single critical point $\alpha$ of multiplicity $d$, then $|S|=|S_\alpha|$ and the result is immediate from  Theorem \ref{maintheoremp2} $(iii)$. This takes care of the case $d=1$.  If $d=2$ and there are two critical points each of multiplicity one, then by part $(ii)$ of Theorem \ref{maintheoremp2},
$
|S| \le 2 \cdot 2^{\frac {m+t+1}2} < 2^{\frac 53}\, 2^{\frac t{3}} 2^{m(1-\frac 1{3})}.
$

Suppose finally that $d=3$ and that there is
 one critical point of multiplicity 2 and one of multiplicity 1. Then $|S| \le 2^{\frac 53}\, 2^{\frac t3}2^{\frac 23 m}+ 2^{\frac 12+\frac t2 +\frac m2}<2^{\frac 53}\,  2^{\frac t4}2^{\frac 34 m}$ (as desired) provided
that $2^{\frac 53} \, 2^{\frac {t-m}{12}} +2^{\frac 12+\frac {t-m}4} \le 2^{\frac 53}$, which is the case for $m-t \ge 5$.  If $m-t \le 4$ then $2^{\frac {m-t}{d+1}} \le 2<2^{\frac 53}$ and so as noted in \eqref{2trivial}, the  bound holds trivially.

\section{Proof of Theorem \ref{maintheorem0} and Corollary \ref{maincor} for $p=2$}

Let $f,g$ be rational functions over $\mathbb Z$,  
 and
$\chi$ a multiplicative character mod $2^m$ with $m \ge 2$.
Suppose that the sum
$S:=S(\chi,g,f,2^m)$ does not degenerate to one of smaller modulus,
that is, either $\deg_2(f)\ge 1$ or, $\chi$ is primitive and $\deg_2(g) \ge
1$. 
If $D=1$,
 then as for the case of odd $p$ we have $S(\chi,g,f,2^m)=0$ 
 Thus, we may assume that
$D \ge 2$.

Suppose that $\deg_2(f) \ge 1$. 
For $m \le t+2$, using $2^t \le \deg_p(f)$ we have trivially
$$
|S| \le 2^{\frac mD}2^{m(1-\frac 1D)}\le 2^{\frac 2D} 2^{\frac tD} 2^{m(1-\frac 1D)}\le 2 \deg_2(f)^{\frac 1D} 2^{m(1-\frac 1D)}.
$$
If $m \ge t+3$, we obtain from Corollary \ref{maincorp21},  
$$
|S| \le 2^{\frac 53} 2^{\frac t{D}} 2^{m(1-\frac 1D)} \le 2^{\frac 53} \deg_2(f)^{\frac 1D} 2^{m(1-\frac 1D)}.
$$
Suppose now that $\deg_2(g) \ge 1$ so that $2^t \le \deg_p(g)$. As above we obtain for $m \le t+2$,
$$
|S| \le 2^{\frac mD}2^{m(1-\frac 1D)}\le 2^{\frac 2D} 2^{\frac tD} 2^{m(1-\frac 1D)}\le 2 \deg_2(g)^{\frac 1D} 2^{m(1-\frac 1D)}.
$$
If $m \ge t+3$, then
$$
|S| \le 2^{\frac 53} 2^{\frac t{D}} 2^{m(1-\frac 1D)} \le 2^{\frac 53} \deg_p(g)^{\frac 1D} 2^{m(1-\frac 1D)}.
$$

\begin{proof}[Proof of Corollary \ref{maincor} for $p=2$]
If $\deg_2(f) \ge 1$, then 
$$
|S| \le 2^{\frac 53}D^{\frac 1D}2^{m(1-\frac 1D)} \le 2^{\frac 53} 3^{\frac 13} 2^{m(1-\frac 1D)}.
$$
Suppose now that $\deg_2(g) \ge 1$. If $m \le t+2$ then
$$
|S| \le 2^{\frac {t+2}\Delta} 2^{m(1-\frac 1\Delta)}\le 2^{\frac 2\Delta} \deg_2(g)^{\frac 1\Delta} 2^{m(1-\frac 1\Delta)} \le  (4\Delta)^{\frac 1\Delta} 2^{m(1-\frac 1\Delta)}\le 2^{\frac 32} 2^{m(1-\frac 1\Delta)}.
$$
If $m \ge t+3$, then
$$
|S| \le  2^{\frac 53} 2^{\frac t{D}} 2^{m(1-\frac 1D)} \le 2^{\frac 53} 2^{\frac t{\Delta}} 2^{m(1-\frac 1\Delta)}\le 2^{\frac 53} \Delta^{\frac 1\Delta}2^{m(1-\frac 1\Delta)} \le   2^{\frac 53} 3^{\frac 13}2^{m(1-\frac 1\Delta)}.
$$

\end{proof}

\section{Proofs of Proposition \ref{degenprop}, Theorem  \ref{degentheorem} and Theorem \ref{Laurenttheorem} for $p=2$}

The proofs of these results  follow the proofs for odd $p$ so we will be brief.
  Start with Proposition \ref{degenprop}. For cases with $\ell_g=0$ the proof is identical.
 Suppose now that $\ell_g>0$ and $\ell_f>0$ and that $\chi$ is a character mod $2^m$.
 Define $H$ as in \eqref{Hdef1},
\begin{equation} \label{Hdef22}
H(X):= 2^{\ell_f}F(X) + c_\chi \log(1+2^{\ell_g}G(X))= f(X)-f(0)+c_\chi\log(g(X)/g(0)),
\end{equation} 
and say $2^\ell\|H(X)$.
Again we have $2^{t-\ell}| (2^{-\ell}H)'$, and so by Corollary \ref{tboundcor},
\begin{equation} \label{2tlub}
2^{t-\ell} \le \deg_2(2^{-\ell} H):= d_2.
\end{equation}
For $m \ge t+3$, it follows from Corollary \ref{maincor2p2} that
$$
|S(\chi,g,f,2^m)| \le 2^{\frac 53} 2^{\frac t{D}} 2^{m(1-\frac 1D)} \le 2^{\frac 53} 2^{\frac {\ell}D}d_2^{\frac 1D} 2^{m(1-\frac 1D)}.
$$
For $m \le t+2$, we have trivially
$$  
|S(\chi,g,f,2^m)| \le 2^{\frac mD}2^{m(1-\frac 1D)}\le  2^{\frac {t+2}D} 2^{m(1-\frac 1D)} \le 2^{\frac {\ell+2}D}d_2^{\frac 1D} 2^{m(1-\frac 1D)}.
$$
and the result follows since $\frac 2D \le \frac 53$ for $D \ge 2$.

The proofs of Theorem \ref{degentheorem} and Theorem \ref{Laurenttheorem} follow identically as for the case of odd $p$.

\end{small}
\end{document}